\documentclass[a4paper,12pt]{article}

\usepackage[utf8]{inputenc}

\usepackage{mathrsfs}
\usepackage{graphicx}
\usepackage{enumerate}
\usepackage{multicol}
\usepackage{color}
\usepackage{amsmath,amssymb,amscd}
\usepackage{amsthm}

\usepackage{pdfpages}
\usepackage{hyperref}

\usepackage{times}
\usepackage[varg]{txfonts}

\theoremstyle{plain}
\newtheorem{thm}{Theorem}
\newtheorem{lem}{Lemma}
\newtheorem{cor}{Corollary}
\newtheorem{prop}{Proposition}

\theoremstyle{definition}
\newtheorem{dfn}{Definition}
\newtheorem{ex}{Example}

\theoremstyle{remark}
\newtheorem{rem}{Remark}
\newtheorem{opp}{Open Problem}
\newcommand{\ph}{\varphi}

\newcommand{\C}{\mathbb{C}}
\newcommand{\K}{\mathbb{K}}

\title{Characterization of field homomorphisms through Pexiderized functional equations}
\author{Eszter Gselmann, Gergely Kiss and Csaba Vincze}

\begin{document}

\maketitle

\begin{abstract}
  The aim of this paper is to prove characterization theorems for field homomorphisms. More precisely,
  the main result investigates the following problem.
  Let $n\in \mathbb{N}$ be arbitrary, $\mathbb{K}$ a field and
  $f_{1}, \ldots, f_{n}\colon \mathbb{K}\to \mathbb{C}$ additive functions.
  Suppose further that equation
\[
 \sum_{i=1}^{n}f^{q_{i}}_{i}\left(x^{p_{i}}\right)=0
 \qquad
 \left(x\in \mathbb{K}\right)
\]
is also satisfied. Then the functions $f_{1}, \ldots, f_{n}$ are
linear combinations of field homomorphisms from $\mathbb{K}$ to
$\mathbb{C}$.
\end{abstract}

\begin{center}
 \emph{Dedicated to the $70$\textsuperscript{th} birthday of Professor Miklós Laczkovich}
\end{center}

\section{Introduction}
The main purpose of this work is to put the previous investigations
into a unified framework and to prove characterization theorems for
field homomorphisms. The problem to be studied reads as follows.

Let $n\in \mathbb{N}$ be arbitrary, $\mathbb{K}$ a field and let
$f_{1}, \ldots, f_{n}\colon \mathbb{K}\to \mathbb{C}$ be additive
functions. Suppose further that we are given natural numbers $p_{1},
\ldots, p_{n}, q_{1}, \ldots, q_{n}$ so that
\[
\tag{$\mathscr{C}$}
 \begin{array}{rcl}
  p_{i}&\neq&p_{j} \qquad \text{ for } \quad i\neq j\\
  q_{i}&\neq&q_{j} \qquad \text{ for } \quad i\neq j\\
  1<p_{i}\cdot q_{i}&=&N\qquad \text{ for } i=1, \ldots, n.
 \end{array}
\]
Suppose also that equation
\begin{equation}\label{Eq1.3}
 \sum_{i=1}^{n}f^{q_{i}}_{i}\left(x^{p_{i}}\right)=0
\end{equation}
is satisfied. 
Throughout this paper we always assume that the field
$\mathbb{K}$ has characteristic 0 (about the problem on other fields
we refer to Open problem 4 in Section \ref{s5}). In what follows, we
show that equation \eqref{Eq1.3} along with condition
$(\mathscr{C})$ is suitable to characterize homomorphisms acting
between the fields $\mathbb{K}$ and $\mathbb{C}$.

\begin{rem}
 Obviously, solving functional equation \eqref{Eq1.3} is meaningful without condition
 $(\mathscr{C})$.  At the same time, we have to point out that without this condition we cannot
 expect in general that all the solutions are linear combinations of homomorphisms, or it can happen that the general problem
 can be reduced to the above formulated problem.

 Indeed, if conditions
 \[
   1<p_{i}\cdot q_{i}=N\qquad \text{ for } i=1, \ldots, n
 \]
are not satisfied, then the homogeneous terms of the same degree can
be collected together, provided that $\mathbb{K}$ is of
characteristic zero (in such a situation we have $\mathbb{Q}\subset
\mathbb{K}$). To show this, assume that
\[
\begin{array}{rcl}
 p_{i}q_{i}&=&N_{1} \quad i=1, \ldots, k_{1}\\
 p_{i}q_{i}&=&N_{2} \quad i=k_{1}+1, \ldots, k_{2}\\
&\vdots&  \\
p_{i}q_{i}&=&N_{j+1} \quad i=k_{j}+1, \ldots, n\\
\end{array}
\]
where the positive integers $N_{1}, \ldots, N_{j+1}$ are different.
Let $r\in \mathbb{Q}$ and $x\in \mathbb{K}$ be arbitrary and
substitute $rx$ in place of $x$ in equation \eqref{Eq1.3} to get
\begin{multline*}
 0
 =
 \sum_{i=1}^{n}f^{q_{i}}_{i}\left((rx)^{p_{i}}\right)
 =
 \sum_{i=1}^{n}r^{p_{i}q_{i}}f^{q_{i}}_{i}\left(x^{p_{i}}\right)
 \\
 =
r^{N_{1}} \sum_{i=1}^{k_{1}}f^{q_{i}}_{i}\left(x^{p_{i}}\right)+
r^{N_{2}}
\sum_{i=k_{1}+1}^{k_{2}}f^{q_{i}}_{i}\left(x^{p_{i}}\right)+ \cdots+
r^{N_{j+1}} \sum_{i=k_{j}+1}^{n}f^{q_{i}}_{i}\left(x^{p_{i}}\right).
\end{multline*}
Observe that the right hand side of this identity is a polynomial of
$r$ for any fixed $x\in \mathbb{K}$, that has infinitely many zeros.
This yields however that this polynomial cannot be nonzero,
providing that all of its coefficients have to be zero, i.e.,
\[
 \begin{array}{rcl}
\displaystyle\sum_{i=1}^{k_{1}}f^{q_{i}}_{i}\left(x^{p_{i}}\right)&=&0 \\[3mm]
\displaystyle \sum_{i=k_{1}+1}^{k_{2}}f^{q_{i}}_{i}\left(x^{p_{i}}\right)&=&0 \\[3mm]
&\vdots& \\
\displaystyle\sum_{i=k_{j}+1}^{n}f^{q_{i}}_{i}\left(x^{p_{i}}\right)&=&0
 \end{array}
\]
This means that in such a situation the original problem can be
split into several problems, where condition $(\mathscr{C})$ already
holds.

On the other hand,  if condition
\[
 \begin{array}{rcl}
   p_{i}&\neq&p_{j} \qquad \text{ for } \quad i\neq j\\
  q_{i}&\neq&q_{j} \qquad \text{ for } \quad i\neq j\\
 \end{array}
\]
is not satisfied then in general we cannot expect that the solutions
are linear combinations of field homomorphisms. Namely, in such a
situation \emph{arbitrary} additive functions can occur as solution,
even in the simplest cases.

To see this, let $p, q\in \mathbb{N}$ be arbitrarily fixed and let
$a\colon \mathbb{K}\to \mathbb{C}$ be an \emph{arbitrary} additive
function. Furthermore, assume that for the complex constants
$\alpha_{1}, \ldots, \alpha_{n}$, identity
\[
 \alpha_{1}^{q}+\cdots+\alpha_{n}^{q}=0
\]
holds and consider the additive functions
\[
 f_{i}(x)=\alpha_{i}a(x)
 \qquad
 \left(x\in \mathbb{K}\right).
\]
Clearly, equation
\[
 \sum_{i=1}^{n}f_{i}(x^{p})^{q}=0
\]
is fulfilled for all $x\in \mathbb{K}$. At the same time, in general
we cannot state that any of these functions is a linear combination
of field homomorphisms.
\end{rem}

\section{Theoretical background}

In this section we collect some results concerning multiadditive
functions, polynomials and exponential polynomials and differential
operators. This collection highlights the main theoretical ideas
that we follow subsequently.

\subsection{The symmetrization method}
\begin{dfn}
 Let $G, S$ be commutative semigroups, $n\in \mathbb{N}$ and let $A\colon G^{n}\to S$ be a function.
 We say that $A$ is \emph{$n$-additive} if it is a homomorphism of $G$ into $S$ in each variable.
 If $n=1$ or $n=2$ then the function $A$ is simply termed to be \emph{additive}
 or \emph{biadditive}, respectively.
\end{dfn}

The \emph{diagonalization} or \emph{trace} of an $n$-additive
function $A\colon G^{n}\to S$ is defined as
 \[
  A^{\ast}(x)=A\left(x, \ldots, x\right)
  \qquad
  \left(x\in G\right).
 \]
As a direct consequence of the definition each $n$-additive function
$A\colon G^{n}\to S$ satisfies
\[
 A(x_{1}, \ldots, x_{i-1}, kx_{i}, x_{i+1}, \ldots, x_n)
 =
 kA(x_{1}, \ldots, x_{i-1}, x_{i}, x_{i+1}, \ldots, x_{n})
 \qquad
 \left(x_{1}, \ldots, x_{n}\in G\right)
\]
for all $i=1, \ldots, n$, where $k\in \mathbb{N}$ is arbitrary. The
same identity holds for any $k\in \mathbb{Z}$ provided that $G$ and
$S$ are groups, and for $k\in \mathbb{Q}$, provided that $G$ and $S$
are linear spaces over the rationals. For the diagonalization of $A$
we have
\[
 A^{\ast}(kx)=k^{n}A^{\ast}(x)
 \qquad
 \left(x\in G\right).
\]

One of the most important theoretical results concerning
multiadditive functions is the so-called \emph{Polarization
formula}, that briefly expresses that every $n$-additive symmetric
function is \emph{uniquely} determined by its diagonalization under
some conditions on the domain as well as on the range. Suppose that
$G$ is a commutative semigroup and $S$ is a commutative group. The
action of the {\emph{difference operator}} $\Delta$ on a function
$f\colon G\to S$ is defined by the formula
\[\Delta_y f(x)=f(x+y)-f(x);\]
note that the addition in the argument of the function is the
operation of the semigroup $G$ and the subtraction means the inverse
of the operation of the group $S$.

\begin{thm}[Polarization formula]\label{Thm_polarization}
 Suppose that $G$ is a commutative semigroup, $S$ is a commutative group, $n\in \mathbb{N}$ and $n\geq 1$.
 If $A\colon G^{n}\to S$ is a symmetric, $n$-additive function, then for all
 $x, y_{1}, \ldots, y_{m}\in G$ we have
 \[
  \Delta_{y_{1}, \ldots, y_{m}}A^{\ast}(x)=
  \left\{
  \begin{array}{rcl}
   0 & \text{ if} & m>n \\
   n!A(y_{1}, \ldots, y_{m}) & \text{ if}& m=n.
  \end{array}
  \right.
 \]

\end{thm}

\begin{cor}
 Suppose that $G$ is a commutative semigroup, $S$ is a commutative group, $n\in \mathbb{N}$ and $n\geq 1$.
 If $A\colon G^{n}\to S$ is a symmetric, $n$-additive function, then for all $x, y\in G$
 \[
  \Delta^{n}_{y}A^{\ast}(x)=n!A^{\ast}(y).
\]
\end{cor}

\begin{lem}
\label{mainfact}
  Let $n\in \mathbb{N}$, $n\geq 1$ and suppose that the multiplication by $n!$ is surjective in the commutative semigroup $G$ or injective in the commutative group $S$. Then for any symmetric, $n$-additive function $A\colon G^{n}\to S$, $A^{\ast}\equiv 0$ implies that
  $A$ is identically zero, as well.
\end{lem}

The polarization formula plays the central role in the investigation
of functional equations characterizing homomorphisms.

\subsection{Polynomial and exponential functions}

In what follows $(G, \cdot)$ is assumed to be  a commutative group.

\begin{dfn}
{\it Polynomials} are elements of the algebra generated by additive
functions over $G$. 
Namely, if $n$ is a positive integer, $P\colon\mathbb{C}^{n}\to \mathbb{C}$ is a (classical) complex polynomial in
 $n$ variables and $a_{k}\colon G\to \mathbb{C}\; (k=1, \ldots, n)$ are additive functions, then the function
 \[
  x\longmapsto P(a_{1}(x), \ldots, a_{n}(x))
 \]
is a polynomial and, also conversely, every polynomial can be
represented in such a form.
\end{dfn}

\begin{rem}
 We recall that elements of $\mathbb{N}^{n}$ for any positive integer $n$ are called
 ($n$-dimensional) \emph{multi-indices}.
 Addition, multiplication and inequalities between multi-indices of the same dimension are defined component-wise.
 Further, we define $x^{\alpha}$ for any $n$-dimensional multi-index $\alpha$ and for any
 $x=(x_{1}, \ldots, x_{n})$ in $\mathbb{C}^{n}$ by
 \[
  x^{\alpha}=\prod_{i=1}^{n}x_{i}^{\alpha_{i}}
 \]
where we always adopt the convention $0^{0}=0$. We also use the
notation $\left|\alpha\right|= \alpha_{1}+\cdots+\alpha_{n}$. With
these notations any polynomial of degree at most $N$ on the
commutative semigroup $G$ has the form
\[
 p(x)= \sum_{\left|\alpha\right|\leq N}c_{\alpha}a(x)^{\alpha}
 \qquad
 \left(x\in G\right),
\]
where $c_{\alpha}\in \mathbb{C}$ and $a\colon G\to \mathbb{C}^{n}$
is an additive function. Furthermore, the \emph{homogeneous
term of degree $k$} of $p$ is
\[
 \sum_{\left|\alpha\right|=k}c_{\alpha}a(x)^{\alpha} .
\]
\end{rem}

\begin{lem}[Lemma 2.7 of \cite{Sze91}]\label{L_lin_dep}
 Let $G$ be a commutative group,
 $n$ be a positive integer and let
 \[
  a=\left(a_{1}, \ldots, a_{n}\right),
 \]
where $a_{1}, \ldots, a_{n}$ are linearly independent complex valued
additive functions defined on $G$. Then the monomials
$\left\{a^{\alpha}\right\}$ for different multi-indices are linearly
independent.
\end{lem}

\begin{dfn}
A function $m\colon G\to \mathbb{C}$ is called an \emph{exponential}
function if it satisfies
\[
 m(xy)=m(x)m(y)
 \qquad
 \left(x,y\in G\right).
\]
Furthermore, on an  \emph{exponential polynomial} we mean a linear
combination of functions of the form $p \cdot m$, where $p$ is a
polynomial and $m$ is an exponential function.
\end{dfn}

It is worth to note that an exponential function is either nowhere
zero or everywhere zero.

The following lemma will be useful in the proof of Theorem
\ref{thm_sep}.
\begin{lem}[Lemma 6. of \cite{L14}]\label{lem_tils}
 Let $G$ be an Abelian group, and let $V$ be a translation invariant
linear subspace of all complex-valued functions defined on $G$. 
Suppose that $\sum_{i=1}^n p_i \cdot m_i \in V$, where $p_1,\ldots, p_n :G\to \mathbb{C}$ are
nonzero polynomials and $m_1,\ldots,m_n:G\to \mathbb{C}$ are distinct
exponentials for every $i = 1,\ldots,n$. Then $p_i \cdot m_i \in V$ and
$m_i \in V$ for every $i = 1,\ldots,n$.
\end{lem}

\subsubsection{Algebraic independence}

As a remarkable ingredient of our argument, we recall a theorem of
Reich and Schwaiger \cite{SchRei84}. The original statement was
formulated for functions defined on $\mathbb{C}$ (with respect to addition).

\begin{thm}\label{thm_indep}
Let $k,l,N$ be positive integers such that $k,l\le N$. Let $m_1, \dots, m_k\colon \mathbb{C}\to \mathbb{C}$ be distinct nonconstant exponential functions, $a_1, \dots, a_l\colon \mathbb{C}\to \mathbb{C}$ additive functions that are linearly independent over $\mathbb{C}$. Then the functions $m_1, \dots, m_k, a_1, \dots, a_l$ are algebraically independent over $\C$. \\
In particular, Let $P_s\colon \mathbb{C}^l\to \mathbb{C}$ be a
classical complex polynomial of $l$ variables for all multi-index
$s$ satisfying  $|s|\le N$. Then the identity
\begin{equation}\label{eqindep}
    \sum_{s\colon |s|\le N } P_{s}(a_1, \dots, a_l) m_1^{s_1}\cdots m_k^{s_k}=0
\end{equation} implies that all polynomials $P_s$ vanish identically $(|s|\le N)$.
\end{thm}

Now we just focus on the last part of the statement. Most of the original
argument works without changes for functions defined on any Abelian
group. 
For an arbitrary field $\mathbb{K}$ we denote  $\mathbb{K}^{\times}$ (resp.
$\mathbb{K}^{+}$) the \emph{ multiplicative (resp. additive) group}
of $\mathbb{K}$.

\begin{enumerate}[(A)]
 \item Let $G$ be an Abelian group.
 If the additive functions $a_1, \dots, a_l: G \to \mathbb{C}$ are linearly independent over $\mathbb{C}$,
 then any system of terms $a_1^{s_1}\cdots a_l^{s_l}$ are also linearly independent over $\mathbb{C}$
 for different nonzero multi-indices $(s_1, \dots, s_l)\in \mathbb{N}^l$. Note that $s_1=\dots=s_l=0$ provides the constant functions.
 This statement is nothing but Lemma \ref{L_lin_dep}.
\item Nonconstant exponentials $m_1, \dots, m_k: G\to \mathbb{C}$ are algebraically independent 
if and only if $m_1^{s_1 }\cdots m_k^{s_k}\ne 1$ for any
$(s_1, \dots, s_k)\in \mathbb{N}^k$. The latter is not necessarily
holds in general. Indeed, for the $n$-ordered cyclic group
$\mathbb{Z}_{n}$ (with respect to addition) the statement is not true since
$\varphi^n\equiv 1$ for every character $\varphi: \mathbb{Z}_{n}\to
\mathbb{C}$. 

In our case, when $G=\mathbb{K}^{\times}$ and the
functions are additive on $\mathbb{K}^+$ the analogue holds.
Obviously, exponential functions on $\mathbb{K}^{\times}$ that are
additive on $\mathbb{K}^+$ are the field homomorphisms of
$\mathbb{K}$. Therefore none of them are constant.

Let $\varphi_1, \dots, \varphi_k$ be field homomorphisms. To show
that $\varphi_1^{s_1 }\cdots \varphi_k^{s_k}\ne 1$ for any nonzero multi-index 
$(s_1,\dots, s_k)\in \mathbb{N}^k$ is enough to find a witness element
$h\ne 0\in \mathbb{K}$ such that $\varphi_1^{s_1 }\cdots
\varphi_k^{s_k}(h)\ne 1$. As a special case ($J'=\emptyset$) we get
it from the following statement.
\begin{lem}(\cite[Lemma 3.3]{KisVar13})
 Let $\mathbb{K}$ be a field of characteristic 0, let $\varphi_1, \dots , \varphi_k\colon \mathbb{K}\to \mathbb{C}$ be distinct homomorphisms
for a positive integer $k$. Then there exists an element $0\ne h \in
\mathbb{K}$ such that
$$\prod_{j\in J}
\varphi_j (h) \ne \prod_{j' \in J'} \varphi_{j'}(h),$$ whenever $J$
and $J'$ are distinct multisets of the elements $1, \dots , k$.
\end{lem}

\item
Combining these facts and using Lemma \ref{L_lin_dep} or following the argument of \cite[Theorem
6.]{SchRei84} we get that if $a_1,\dots, a_l$ are linearly
independent and $m_1, \dots, m_k$ are nonconstant exponential
functions, then equation \eqref{eqindep} holds if and only if every
\allowbreak $P_{s}(a_1, \dots, a_l)\cdot \allowbreak m_1^{s_1}\cdots
m_k^{s_k}=0$ for all $s=(s_1, \dots, s_k)$, $|s|\leq N$.
\end{enumerate}
Applying (A)-(C) we get the following statement.

\begin{thm}\label{thm_indep2}
Let $\mathbb{K}$ be a field of characteristic 0 and $k,l,N$ be
positive integers such that $k,l\le N$. Let $m_1, \dots, m_k\colon
\mathbb{K}^{\times}\to \mathbb{C}$ be distinct exponential functions
that are additive on $\mathbb{K}^+$, let $a_1, \dots, a_l\colon
\mathbb{K}^{\times}\to \mathbb{C}$ be additive functions that are
linearly independent over $\mathbb{C}$ and let $ P_s\colon
\mathbb{C}^l\to \mathbb{C}$ be classical complex polynomials of $l$
variables for all $|s|\leq N$. Then the equation
\begin{equation}\label{eqindep2}
    \sum_{s\colon |s|\le N } P_{s}(a_1, \dots, a_l) m_1^{s_1}\cdots m_k^{s_k}=0
\end{equation}
implies that all polynomials $P_s$ vanish identically $(|s|\le N)$.
\end{thm}


\subsection{ Levi-Civit\`{a}  equations}

As we will see in the next section, the so-called Levi-Civit\`{a}
functional equation will have a distinguished role in our
investigations. Thus, below the most important statements will be
summarized. Here we follow the notations and the terminology of
L.~Sz\'e\-kely\-hi\-di \cite{Sze91}, \cite{Sze06}.

\begin{thm}[Theorem 10.1 of \cite{Sze91}]
 Any finite dimensional translation invariant linear space of continuous complex valued functions on a topological Abelian group is spanned by 
 exponential polynomials.
\end{thm}

In view of this theorem, if $(G, \cdot)$ is an Abelian group, then
any function $f\colon G\to \mathbb{C}$ satisfying the so-called
\emph{Levi-Civit\`{a} functional equation}, that is,
\begin{equation}\label{Eqlevi}
 f(x\cdot y)= \sum_{i=1}^{n}g_{i}(x)h_{i}(y)
 \qquad
 \left(x, y\in G\right)
\end{equation}
for some positive integer $n$ and functions $g_{i}, h_{i}\colon G\to
\mathbb{C}\; (i=1, \ldots, n)$, is an exponential polynomial of
order at most $n$. Indeed, equation \eqref{Eqlevi} expresses the fact that 
all the translates of the function $f$ belong to the same finite dimensional translation invariant linear space, namely 
\[
 \tau_{y}f \in \mathrm{lin} \left(g_{1}, \ldots, g_{n}\right) 
 \]
holds for all $y\in G$. 

Obviously, if the functions $h_{1}, \ldots, h_{n}$ are linearly
independent, then $g_{1}, \ldots, g_{n}$ are linear combinations of
the translates of $f$, hence they are exponential  polynomials of
order at most $n$, too. Moreover, they are built up from the same
additive and exponential functions as the function $f$.

Before presenting the solutions of equation \eqref{Eqlevi}, we
introduce some notions.

\begin{rem}
Let $k, n, n_{1}, \ldots, n_{k}$ be positive integers with
$n=n_{1}+\cdots+n_{k}$ and let for $j=1, \ldots, k$ the complex
polynomials $P_{j}, Q_{i, j}$ of $n_{j}-1$ variables and of degree
at most $n_{j}-1$ be given, $i=1, \ldots, n; j=1, \ldots, k$. For
any $j=1, \ldots, k$ and for arbitrary multi-indices
$I_{j}=\left(i_{1}, \ldots, i_{n_{j}-1}\right)$ and
$J_{j}=\left(j_{1}, \ldots, j_{n_{j}-1}\right)$ we define the
$n_{j}\times n_{j}$ matrix $M_{j}(P; I_{j}, J_{j})$ and the
$n_{j}\times n$ matrix $N_{j}(Q; I_{j})$ as follows: for any choice
of $p, q=0, 1, \ldots, n_{j}-1$ the $(n_{j}-p, n_{j}-q)$ element of
$M_{j}(P; I_{j}, J_{j})$ is given by
\[
 M_{j}(P; I_{j}, J_{j})_{(n_{j}-p, n_{j}-q)}=
 \begin{cases}
\frac{1}{p!q!}\partial_{i_{1}}\cdots \partial_{i_{p}}\partial_{j_{1}}\cdots\partial_{j_{q}}P_{j}\left(0, \ldots, 0\right)& \text{for } p+q< n_{j}\\
0 & \text{otherwise}
 \end{cases}
\]
and for any choice of $p=1, 2, \ldots, n_{j}$, $q=1, 2, \ldots, n$
the $(p, q)$ element of $N_{j}(Q; I_{j})$ is given by
\[
 N_{j}(Q; I_{j})_{p, q}=\dfrac{1}{(n_{j}-p)!}\partial_{i_{1}}\cdots \partial_{i_{n_{j}-p}}Q_{q, p}(0, \ldots, 0).
\]
Then let us define the $n\times n$ block matrices $M\left(P; I_{1},
\ldots, I_{k}, J_{1}, \ldots, J_{k}\right)$ and $N\left(Q; I_{1},
\ldots, I_{k}\right)$ by
\[
 M\left(P; I_{1}, \ldots, I_{k}, J_{1}, \ldots, J_{k}\right)
 =
 \begin{pmatrix}
\fbox{$M_{1}(P, I_{1}, J_{1})$} & 0 & \ldots & 0\\
 0 & \fbox{$M_{2}(P, I_{2}, J_{2})$} & 0 & \ldots \\
\vdots & 0  & \ddots& \vdots\\
\vdots& \vdots & & \fbox{$M_{k}(P, I_{k}, J_{k})$}
 \end{pmatrix}
\]
and
\[
N\left(Q; I_{1}, \ldots, I_{k}\right) =
\begin{pmatrix}
\fbox{$ N_{1}(Q; I_{1})$}\\
 \vdots\\
 \fbox{$N_{k}(Q; I_{k})$}
\end{pmatrix}.
\]

\end{rem}

The idea of using Levi-Civit\`a equations rely on Theorem 10.4 of
\cite{Sze91} which is the following.

\begin{thm}\label{Szekely}
 Let $G$ be an Abelian group, $n$ be a positive integer and $f, g_{i}, h_{i}\colon G\to \mathbb{C}\; (i=1, \ldots, n)$ be functions so that both the sets
 $\left\{g_{1}, \ldots, g_{n}\right\}$ and $\left\{h_{1}, \ldots, h_{n}\right\}$ are linearly independent.
 The functions  $f, g_{i}, h_{i}\colon G\to \mathbb{C}\; (i=1, \ldots, n)$ form a \emph{non-degenerate} solution of equation \eqref{Eqlevi}
 if and only if
 \begin{enumerate}[(a)]
  \item there exist positive integers $k, n_{1}, \ldots, n_{k}$ with $n_{1}+\cdots+n_{k}=n$;
  \item there exist different nonzero complex exponentials $m_{1}, \ldots, m_{k}$;
  \item for all $j=1, \ldots, k$ there exists linearly independent sets of complex additive functions \[\left\{a_{j, 1}, \ldots,  a_{j, n_{j}-1}\right\};\]
  \item there exist polynomials $P_{j}, Q_{i, j}, R_{i, j}\colon \mathbb{C}^{n_{j}-1}\to \mathbb{C}$ for all $i=1, \ldots, n; j=1, \ldots, k$ in
  $n_{j}-1$ complex variables and of degree at most $n_{j}-1$;
 \end{enumerate}
 so that we have
 \[
  f(x)= \sum_{j=1}^{k}P_{j}\left(a_{j, 1}(x), \ldots, a_{j, n_{j}-1}(x)\right)m_{j}(x)
 \]
 \[
  g_{i}(x)= \sum_{j=1}^{k}Q_{i, j}\left(a_{j, 1}(x), \ldots, a_{j, n_{j}-1}(x)\right)m_{j}(x)
 \]
 and
 \[
  h_{i}(x)= \sum_{j=1}^{k}R_{i, j}\left(a_{j, 1}(x), \ldots, a_{j, n_{j}-1}(x)\right)m_{j}(x)
 \]
for all $i=1, \ldots, n$. Furthermore,
\[
 M\left(P; I_{1}, \ldots, I_{k}, J_{1}, \ldots, J_{k}\right)
 =
 N\left(Q; I_{1}, \ldots, I_{k}\right)N\left(R; J_{1}, \ldots, J_{k}\right)^{T}
\]
holds for any choice of the multi-indices $I_{j}, J_{j}\in
\mathbb{N}^{n_{j}-1}\, \left(j=1, \ldots, k\right)$, here $^{T}$
denotes the transpose of a matrix.
\end{thm}



In \cite{Shu10} E.~Shulman used some techniques and results from
representation theory to investigate a multivariate extension of the
Levi-Civit\`{a} equation. In order to quote her results, we need the
following notions.

\begin{rem}
 The notion of exponential polynomials can be formulated not only in the framework of the theory of functional equations
 but also in that of representation theory. This point of view can be really useful in many cases.
 Let $G$ be a (not necessarily commutative) topological group and $\mathscr{C}(G)$ be the set of all
 continuous complex valued functions on $G$.
 A function $f\in \mathscr{C}(G)$ is called an exponential
 polynomial function (or a \emph{matrix function}) if there is a
 continuous representation $\pi$ of $G$ on a finite-dimensional topological space $X$
 such that
 \[
  f(g)= \langle \pi(g)x, y \rangle
  \qquad
  \left(g\in G\right),
 \]
where $x\in X$ and $y\in X^{\ast}$.

The minimal dimension of such representations is called the
\emph{degree} or the \emph{order} of the exponential polynomial.

Furthermore, $f\in \mathscr{C}(G)$ is an exponential polynomial of
degree less that $n$ if it is contained in an invariant subspace
$\mathscr{L}\subset \mathscr{C}(G)$ with
$\dim\left(\mathscr{L}\right)\leq n$.
\end{rem}

\begin{dfn}
Let $G$ be a group.
 We say that $f\colon G\to \mathbb{C}$ is a \emph{local exponential polynomial} if its restriction
 to any finitely generated subgroup $H\subset G$ is an exponential polynomial on $H$.

 A function $f\in \mathscr{C}(G)$ is an \emph{almost exponential polynomial} if for any
 finite subset $E$ of $G$, there is a finite-dimensional subspace $\mathscr{L}_{E}\subset \mathscr{C}(G)$,
 containing $f$ and invariant for all operators $\tau_{g}$ as $g$ runs through $E$,
 where
 \[
  \tau_{g}f(h)= f(hg)
  \qquad
  \left(h\in G\right).
 \]
\end{dfn}

\begin{rem}\label{rem_almost}
 It is an immediate consequence of the above definitions that any exponential polynomial is an
 almost exponential polynomial.
 Furthermore, if $f$ is an almost exponential polynomial, then it is a local exponential polynomial, too.
 Clearly, for finitely generated topological  groups all these three notions coincide.
 At the same time, in general these notions are different, even in case of discrete commutative groups,
 see \cite{Shu10}.
\end{rem}

\begin{dfn}
 Let $G$ be a group and $n\in \mathbb{N}, n\geq 2$.
 A function $F\colon G^{n}\to \mathbb{C}$ is said to be
 \emph{decomposable} if it can be written as a finite sum of products
 $F_{1}\cdots F_{k}$, where all $F_{i}$ depend on disjoint sets of variables.
\end{dfn}

\begin{rem}
 Without the loss of generality we can suppose that $k=2$ in the above definition, that is,
 decomposable functions are those mappings that can be written in the form
 \[
  F(x_{1}, \ldots, x_{n})= \sum_{E}\sum_{j}A_{j}^{E}B_{j}^{E}
 \]
where $E$ runs through all non-void proper subsets of $\left\{1,
\ldots, n\right\}$ and for each $E$ and $j$ the function $A_{j}^{E}$
depends only on variables $x_{i}$ with $i\in E$, while $B_{j}^{E}$
depends only on the variables $x_{i}$ with $i\notin E$.
\end{rem}

\begin{thm}\label{Thm_Shulman}
 Let $G$ be a group and $f\in \mathscr{C}(G)$ and
 $n\in \mathbb{N}, n\geq 2$ be fixed. If the mapping
 \[
  G^{n} \ni \left(x_{1}, \ldots, x_{n}\right)\longmapsto f\left(x_{1}\cdots x_{n}\right)
 \]
is decomposable then $f$ is an almost exponential polynomial
function.
\end{thm}




\subsection{Derivations and differential operators}

Similarly as before, $\mathbb{K}$ denotes a field and
$\mathbb{K}^{\times}$ stands for the multiplicative subgroup of
$\mathbb{K}$.

In this subsection we introduce differential operators acting on
fields which have important role in our investigation.

\begin{dfn}
A \emph{derivation} on $\mathbb{K}$ is a map $d:\mathbb{K}\to
\mathbb{K}$ such that equations
\begin{equation}\label{eder1}
d(x+y)=d(x)+d(y) \qquad \text{and} \qquad d(xy)=d(x)y+x d(y)
\end{equation}
are fulfilled for every $x,y\in \mathbb{K}$.

We say that the map $D:\mathbb{K}\to \mathbb{C}$ is a
\emph{differential operator of order $m$} if $D$ can be represented
as
\begin{equation}\label{eqdiff}
D=\sum_{j=1}^M c_{j} d_{j,1} \circ \ldots \circ  d_{j,k_j},
\end{equation}
where  $c_{j}\in \mathbb{C}$ and $d_{i,j}$  are derivations on
$\mathbb{K}$ and $k_j\le m$ which fulfilled as equality for some
$j$. If $k=0$ then we interpret $d_1 \circ  \ldots \circ  d_k$ as
the identity function $id$ on $\mathbb{K}$.
\end{dfn}



\begin{rem}\label{rem_diff}
Since the compositions $d_1 \circ  \ldots \circ  d_k$ span a linear
space over $\C$, without loss of generality we may assume that each
term of \eqref{eqdiff} are linearly independent. Equivalently we may
fix a basis $\mathscr{B}$ of compositions. We also fix that the
identity map is $id$ in $\mathscr{B}$. We note that a differential
operator of order $n$ contains a composition of length $n$.
\end{rem}

If a function $m$ is additive on $\mathbb{K}$ and exponential on
$\mathbb{K}^{\times}$, then $m$ is clearly a field homomorphism. In
our case this can be extended to $\C$ as an automorphism of $\C$ by
\cite[Theorem 14.5.1]{Kuc09}. Now we concentrate on the subfields of
$\mathbb{C}$ that has finite transcendence degree over $\mathbb{Q}$.

\begin{lem}\label{lem_hom}
Let $\mathbb{K\subset \mathbb{C}}$ be field of finite transcendence
degree  and $\varphi: \mathbb{K}\to \mathbb{C}$ an injective
homomorphism. Then there exists an automorphism $\psi$ of
$\mathbb{C}$ such that $\psi|_K = \varphi$.
\end{lem}

Further relations are presented between the exponential  polynomials
defined on $\mathbb{K}^{\times}$ and differential operators on
$\mathbb{K}$. The connection was first realized in \cite{KL1} and
the connection between the degrees and orders was settled in
\cite{KL2}. Clearly every differential operator is additive on
$\mathbb{K}$ and this additional property is a substantial part of
the following statement.

\begin{thm}\label{thm_deriv}
Suppose that the transcendence degree of the field $\mathbb{K}$ over
$\mathbb{Q}$ is finite.  Let $f : \mathbb{K}\to \mathbb{C}$ be
additive, and let $m$ be an exponential on $K^{\times}$. Let
$\varphi$ be an extension of $m$ to $\mathbb{C}$ as an automorphism
of $\mathbb{C}$. Then the following
are equivalent.
\begin{enumerate}[(i)]
\item $f = p \cdot m$ on $\mathbb{K}^{\times}$, where $p$ is a local polynomial on $\mathbb{K}^{\times}$.
\item  $f = p \cdot m$ on $\mathbb{K}^{\times}$, where $p$ is an almost polynomial on $\mathbb{K}^{\times}$.
\item $f = p \cdot m$ on $\mathbb{K}^{\times}$, where $p$ is a polynomial on $\mathbb{K}^{\times}$.
\item There exists a unique differential operator $D$ on $\mathbb{K}$ such that
$f = \varphi \circ D$ on $\mathbb{K}$.
\end{enumerate}
In this case, $p$ is a polynomial of degree $n$ if and only if $D$
is a differential operator of order $n$.
\end{thm}

\begin{proof}
 The equivalence of $(i), (iii)$ and $(iv)$ follows from \cite[Theorem 4.2]{KL1}. Remark \ref{rem_almost} implies the equivalence of $(ii)$ with the others. The last part of the statement follows from \cite[Corollary 1.1.]{KL2}.
\end{proof}



\section{Preparatory statements}

At first glance equation \eqref{Eq1.3} itself seem not really
restrictive for the functions $f_{1}, \ldots, f_{n}$. At the same
time, our results show that these additive functions are in fact
very special, i.e., they are linear combinations of field
homomorphisms from the field $\mathbb{K}$ to $\mathbb{C}$. This is
caused by the additivity assumption on the involved functions, and
this is the property that can effectively be combined with the
theory of (exponential) polynomials on semigroups. More precisely,
with the aid of the following lemma, we will be able to broaden the
number of the variables appearing in equation \eqref{Eq1.3} from one
to $N$.

\begin{lem}\label{L1}
 Let $n\in \mathbb{N}$ be arbitrary, $\mathbb{K}$ a field, $f_{1}, \ldots, f_{n}\colon \mathbb{K}\to \mathbb{C}$ additive functions. Suppose further that we are given natural
numbers $p_{1}, \ldots, p_{n}, q_{1}, \ldots, q_{n}$ such that they
fulfill condition $(\mathscr{C})$. If
\begin{equation}\label{Eq1.4}
 \sum_{i=1}^{n}f^{q_{i}}_{i}\left(x^{p_{i}}\right)=0
\end{equation}
 is satisfied for any $x\in \mathbb{K}$, then we also have
 \begin{equation}
  \sum_{i=1}^{n}\dfrac{1}{N!}\sum_{\sigma\in \mathscr{S}_{N}}f_{i}\left(x_{\sigma(1)}\cdots x_{\sigma(p_{i})}\right)
  \cdots f_{i}\left(x_{\sigma(N-p_{i}+1)}\cdots x_{\sigma(N)}\right)=0
 \end{equation}
for any $x_{1}, \ldots, x_{N}\in \mathbb{K}$, here $\mathscr{S}_{N}$
denotes the symmetric group of order $N$.
\end{lem}

\begin{proof}
 Suppose that $n\in \mathbb{N}$, $\mathbb{K}$ is a field, $f_{1}, \ldots, f_{n}\colon \mathbb{K}\to \mathbb{C}$
are additive functions and define the function $F\colon
\mathbb{K}^{N}\to \mathbb{C}$ through
\[
 F(x_{1}, \ldots, x_{N})=
 \sum_{i=1}^{n}\dfrac{1}{N!}\sum_{\sigma\in \mathscr{S}_{N}}f_{i}\left(x_{\sigma(1)}\cdots x_{\sigma(p_{i})}\right)
  \cdots f_{i}\left(x_{\sigma(N-p_{i}+1)}\cdots x_{\sigma(N)}\right)
  \quad
  \left(x_{1}, \ldots, x_{N}\in \mathbb{K}\right).
\]
It is clear that $F$ is a symmetric function, moreover, due to the
additivity of the functions $f_{1}, \ldots, f_{n}$, it is
$N$-additive. Furthermore, in view of equation \eqref{Eq1.4},
\[
 F(x, \ldots, x)=  \sum_{i=1}^{n}f^{q_{i}}_{i}\left(x^{p_{i}}\right)=0
 \qquad
 \left(x\in \mathbb{K}\right).
\]
Therefore, the polarization formula immediately yields that the
mapping $F$ is identically zero on $\mathbb{K}^{N}$.
\end{proof}

\subsection*{Equation \eqref{Eq1.3} with two unknown functions }

At first we will investigate the case when $n=2$. This case was also
studied by F.~Halter--Koch and L.~Reich in a special situation (when
$n=p$ and $m=q$) in \cite{Hal00, HalRei00, HalRei01}.

\begin{prop}\label{Prop1}
 Let $n, m, p, q\in \mathbb{N}$ be arbitrarily fixed so that $n\cdot m= p\cdot q>1$ and $m\neq p$. Let $\mathbb{K}$ be a field and suppose that for additive functions
 $f, g\colon \mathbb{K}\to \mathbb{C}$ the  functional equation
 \begin{equation}\label{Eq1.5}
  f^{m}\left(x^{n}\right)=g^{p}\left(x^{q}\right)
  \qquad
  \left(x\in \mathbb{K}\right)
 \end{equation}
is fulfilled. Then, and only then there exists a homomorphism
$\varphi\colon \mathbb{K}\to \mathbb{C}$ so that
\[
 f(x)=f(1)\cdot \varphi(x)
 \quad
 \text{and}
 \quad
 g(x)=g(1)\cdot \varphi(x)
\]
furthermore, we also have $f(1)^{m}-g(1)^{p}=0$.
\end{prop}

\begin{proof}

Let $N=n\cdot m=p\cdot q$. According to Lemma \ref{L1} we have that
the symmetric $N$-additive function
  $F\colon \mathbb{K}^{N}\to \mathbb{C}$ defined by
  \begin{multline*}
   F\left(x_{1}, \ldots, x_{N}\right)= \dfrac{1}{N!}\sum_{\sigma\in \mathscr{S}_{N}}\left[f\left(x_{\sigma(1)}\cdots x_{\sigma(n)}\right)\cdots
   f\left(x_{\sigma(N-n+1)}\cdots x_{\sigma(N)}\right) \right.
   \\
   \left.  -g\left(x_{\sigma(1)}\cdots x_{\sigma(q)}\right)\cdots g\left(x_{\sigma(N-q+1)}\cdots x_{\sigma(N)}\right)
   \right]
   \qquad
   \left(x_{1}, \ldots, x_{N}\in \mathbb{K}\right)
  \end{multline*}
is identically zero due to the fact that
\[
 F\left(x, \ldots, x\right)= f^m(x^{n})-g^{p}(x^q)=0
 \qquad
   \left(x\in \mathbb{K}\right).
\]
From this we get $F(1, 1, 1, \ldots, 1)=0$ which implies
\begin{equation}\label{eqfg1}
    f^m(1)-g^p(1)=0.
\end{equation}
By appropriate substitution, $F(x, 1, 1, \ldots, 1)=0$ clearly
follows for any $x\in \mathbb{K}$, or equivalently
\begin{equation}\label{eqfg2}
 f^{m-1}(1)f(x)-g^{p-1}(1)g(x)=0
 \qquad
 \left(x\in \mathbb{K}\right).
\end{equation}

If $g^{p-1}(1)=0$ and $f^{m-1}(1)\neq 0$, then $f\equiv 0$ would
follow, which is impossible. A similar argument shows that
$f^{m-1}(1)=0$ and $g^{p-1}(1)\neq 0$ is also impossible. This means
that either  $g^{p-1}(1)\neq 0$ and $f^{m-1}(1)\neq 0$ or
$g^{p-1}(1)= 0$ and $f^{m-1}(1)=0$.

If $g^{p-1}(1)\neq 0$ and $f^{m-1}(1)\neq 0$ then
\[
 F(x, y, 1, \ldots, 1)=0
 \qquad
 \left(x, y\in \mathbb{K}\right),
\]
implies that there exist constants $c_{1}, c_{2}, d_{1}, d_{2} \in
\mathbb{Q}$ so that $c_1+c_2=d_1+d_2=1$ and $c_1\ne d_1$ (since
$p\ne m$) such that
\begin{equation}\label{eqfg3}
 c_{1}f^{m-1}(1)f(xy)+c_{2}f^{m-2}(1)f(x)f(y)-d_1g^{p-1}(1)g(xy)-d_2g^{p-2}(1)g(x)g(y)=0
 \qquad
 \left(x, y\in \mathbb{K}\right).
\end{equation}
Applying equations \eqref{eqfg1} and \eqref{eqfg2} we get that
\begin{equation*}
     g^{p-2}(1)g(x)g(y)=\frac{(g^{p-1}(1)g(x))(g^{p-1}(1)g(y))}{g^p(1)}=\frac{(f^{m-1}(1)f(x))(f^{m-1}(1)f(y))}{f^{m}(1)}=  f^{m-2}(1)f(x)f(y),
\end{equation*}
 and we can eliminate $g$ from equation \eqref{eqfg3}
 \begin{align*}
 c_{1}f^{m-1}(1)&f(xy)+c_{2}f^{m-2}(1)f(x)f(y)-d_1f^{m-1}(1)f(xy)-d_2f^{m-2}(1)f(x)f(y)=\\&(c_1-d_1)f^{m-1}(1)f(xy)+(c_2-d_2)f^{m-2}(1)f(x)f(y) = 0.
\end{align*}
Since $c_1+c_2=d_1+d_2=1, c_1\ne d_1$ and $f(1)\ne 0$, it follows
that $c_1-d_1=-(c_2-d_2)\ne 0$ and the last expression can be
reduced to
$$ f(1)f(xy)=f(x)f(y).$$
Taking $\varphi(x)=f(x)/f(1)$ for all $x\in \mathbb{K}$, we get that
$\varphi(xy)=\varphi(x)\varphi(y)$ (i.e. $\varphi$ is
multiplicative). Also $\varphi$ is additive since $f$ is additive.
Thus $\varphi$ is an injective homomorphism of $\mathbb{K}$. A
similar argument shows that $g(x)=g(1)\psi(x)$, where $\psi$ is an
injective homomorphism of $\mathbb{K}$.
Substituting this into equation \eqref{Eq1.5}, we get that
\[
f^m(1)\varphi^N=g^p(1)\psi^N.
\]
Using equation \eqref{eqfg1} and a symmetrization process,
$\phi=\psi$ follows and we get
\[
 f(x)=f(1)\varphi(x)
 \quad
 \text{and}
 \quad
 g(x)=g(1)\varphi(x)
 \qquad
 \left(x\in \mathbb{K}\right)
\]
with a certain homomorphism $\varphi\colon \mathbb{K}\to \mathbb{C}$
and
$f(1)^{m}-g(1)^{n}=0$. 




Finally, if  $g(1)^{q-1}= 0$ and $f(1)^{m-1}=0$, then $g(1)=f(1)=0$
and we have two alternatives. Either $f\equiv 0$ and $g\equiv 0$ or
at least one of them is non-identically zero, say $f\not\equiv 0$.

The first case clearly yields a solution to equation \eqref{Eq1.3}.

Now we show that the latter case is not possible.
Without loss of generality we may assume that $m< p$. Then
\[
 0=F(\underbrace{x,\dots x}_m, 1, \ldots, 1)= C\cdot f(x)^m,
\] for some positive constant $C$. Indeed each other summand stemming from $f$ contain at least one term of $f(1)$ in the product, similarly each product of $g$'s contains $g(1)$.
Therefore $f(x)=0$ for all $x\in \mathbb{K}$, contradicting our
assumption.

\end{proof}

\section{Main results}

Firstly we show that every solution of equation \eqref{Eq1.3} is an
almost exponential  polynomial of the group $\mathbb{K}^{\times}$.

\begin{thm}\label{Thm_main}
 Let $n\in \mathbb{N}$ be arbitrary, $\mathbb{K}$ a field,
 $f_{1}, \ldots, f_{n}\colon \mathbb{K}\to \mathbb{C}$ additive functions. Suppose further that we are given natural numbers
 $p_{1}, \ldots, p_{n}, q_{1}, \ldots, q_{n}$ so that they fulfill condition $(\mathscr{C})$.
 If
 \begin{equation}\label{Eq1.6}
  \sum_{i=1}^{n}f_{i}^{q_{i}}\left(x^{p_{i}}\right)=0
 \end{equation}
holds for all $x\in \mathbb{K}$, then the functions $f_{1}, \ldots,
f_{n}\colon \mathbb{K}\to \mathbb{C}$ are almost exponential
polynomials of the Abelian group $\mathbb{K}^{\times}$.
\end{thm}

\begin{proof}
 Suppose that the conditions are satisfied, then due to Lemma \ref{L1}, we have that the mapping
 $F\colon \mathbb{K}^{N}\to \mathbb{C}$ defined by
 \[
  F(x_{1}, \ldots, x_{N})=\sum_{i=1}^{n}\dfrac{1}{N!}\sum_{\sigma\in \mathscr{S}_{N}}f_{i}\left(x_{\sigma(1)}\cdots x_{\sigma(p_{i})}\right)
  \cdots f_{i}\left(x_{\sigma(N-p_{i}+1)}\cdots x_{\sigma(N)}\right)
  \qquad
  \left(x_{1}, \ldots, x_{N}\in \mathbb{K}\right)
 \]
is identically zero.

From this we immediately conclude that for any $x\in \mathbb{K}$
\[
 F(x, 1, \ldots, 1)=0
\]
holds, that is,
\begin{equation}\label{Eq_dep}
 \sum_{i=1}^{n}f_{i}(x)f_{i}^{q_{i}-1}(1)=0
 \qquad
 \left(x\in \mathbb{K}\right).
\end{equation}

Again, due to the fact that $F$ has to be identically zero, we also
have
\[
 F(x, y, 1, \ldots, 1)=0
 \qquad
 \left(x, y\in \mathbb{K}\right),
\]
i.e.,
\begin{equation}\label{Eq_lc}
\sum_{i=1}^{n}\left[c_{i}f_{i}(xy)+d_{i}f_{i}(x)f_{i}(y)\right]=0
\qquad \left(x, y\in \mathbb{K}\right)
\end{equation}
with certain constants $c_{i}, d_{i}\in \mathbb{C}$.

Without the loss of generality we can (and we also do) assume that
the parameters $q_{1}, \ldots, q_{n}$ are arranged in a strictly
increasing order, that is,
\[
 q_{1} < q_{2} < \cdots < q_{n}
\]
holds and  (due to condition $(\mathscr{C})$) we
have that
\[
 p_{1}> p_{2} > \cdots > p_{n}
\]
is also fulfilled.

We will show by induction on $n$ that all the mappings $f_{1},
\ldots, f_{n}$ are almost exponential polynomials. Since the
multiadditive mapping $F$ is identically zero on $\mathbb{K}^{N}$,
we have that
\begin{multline*}
 \sum_{\sigma\in \mathscr{S}_{N}}
 f_{1}\left(x_{\sigma(1)}\cdots x_{\sigma(p_{1})}\right) \cdots f_{1}\left(x_{\sigma(N-p_{1}+1)}\cdots x_{\sigma(N)}\right)
 \\
 =
 - \sum_{i=2}^{n}\sum_{\sigma\in \mathscr{S}_{N}}f_{i}\left(x_{\sigma(1)}\cdots x_{\sigma(p_{i})}\right)
  \cdots f_{i}\left(x_{\sigma(N-p_{i}+1)}\cdots x_{\sigma(N)}\right)
  \qquad
  \left(x_{1}, \ldots, x_{N}\in \mathbb{K}\right)
\end{multline*}
Let us keep all the variables $x_{p_{1}+1}, \ldots, x_{N}$ be fixed,
while the others are arbitrary. Then the above identity yields that
either $f_{1}$ is identically zero or $f_{1}$ is decomposable. Due
to Theorem \ref{Thm_Shulman}, in any cases we have that $f_{1}$ is
an almost exponential polynomial function. Therefore, for any
finitely generated subgroup $H\subset \mathbb{K}^{\times}$, the
function $f_{1}\vert_{H}$ is an exponential polynomial. In other
words for any finitely generated subgroup $H\subset
\mathbb{K}^{\times}$, the mapping $f_{1}\vert_{H}$ is not only
decomposable but also fulfills a certain multivariate
Levi-Civit\`{a} functional equation.

Assume now that there exists a natural number $k$ with $k\leq n-1$
so that all the mappings $f_{1}, \ldots, f_{k}$ are almost
exponential polynomials. Then, again due to the fact that $F\equiv
0$, we have that
\begin{multline*}
 \sum_{\sigma\in \mathscr{S}_{N}}
 f_{k+1}\left(x_{\sigma(1)}\cdots x_{\sigma(p_{k+1})}\right) \cdots f_{k+1}\left(x_{\sigma(N-p_{k+1}+1)}\cdots x_{\sigma(N)}\right)
 \\
 = -
  \sum_{i=1}^{k}\sum_{\sigma\in \mathscr{S}_{N}}f_{i}\left(x_{\sigma(1)}\cdots x_{\sigma(p_{i})}\right)
  \cdots f_{i}\left(x_{\sigma(N-p_{i}+1)}\cdots x_{\sigma(N)}\right)
 \\ -
   \sum_{i=k+2}^{n}\sum_{\sigma\in \mathscr{S}_{N}}f_{i}\left(x_{\sigma(1)}\cdots x_{\sigma(p_{i})}\right)
  \cdots f_{i}\left(x_{\sigma(N-p_{i}+1)}\cdots x_{\sigma(N)}\right)
  \\
  \left(x_{1}, \ldots, x_{N}\in \mathbb{K}\right).
\end{multline*}
Let us keep all the variables $x_{p_{k+1}+1}, \ldots, x_{N}$ be
fixed, while the others are arbitrary. Then, in view of Theorem
\ref{Thm_Shulman}, this equation yields that either $f_{k+1}$ is
identically zero or $f_{k+1}$ is an almost exponential polynomial,
due to the fact that the first summand on the right-hand side is an
almost exponential polynomial by induction, while the other summand
consists only of decomposable terms.
\end{proof}

\begin{rem}
 Note that if
 \[
  f_{i}(x)=a_{i}f(x)
  \qquad
  \left(x\in \mathbb{K}\right)
 \]
holds for all $i=1, \ldots, n$ with certain complex constants
$a_{1}, \ldots, a_{n}$ (assuming that at least one of them is
nonzero), then we immediately get that there exists a homomorphism
$\varphi\colon \mathbb{K}\to \mathbb{C}$ such that
\[
 f_{i}(x)=f_{i}(1)\varphi(x)
 \qquad
 \left(x\in \mathbb{K}\right).
\]
Indeed, in this case equation \eqref{Eq_lc} yields that
\[
 \sum_{i=1}^{n}\left[c_{i}a_{i}f(xy)+d_{i}a_{i}^{2}f(x)f(y)\right]=0
 \qquad
 \left(x, y\in \mathbb{K}\right),
\]
that is, $f$ satisfies the Pexider equation
\[
 \alpha f(xy)= \beta f(x)f(y)
 \qquad
 \left(x, y\in \mathbb{K}\right).
\]
This means that $f$ is a constant multiple of a multiplicative
function. Since $f$ has to be additive too, this multiplicative
function has to be in fact a homomorphism. All in all, we have that
the additive function $f\colon \mathbb{K}\to \mathbb{C}$ fulfills
equation
\[
 \sum_{i=1}^{n}a_{i}^{q_{i}}f\left(x^{p_{i}}\right)^{q_{i}}=0
 \qquad
 \left(x\in \mathbb{K}\right)
\]
with certain complex constants $a_{1}, \ldots, a_{n}$ if and only if
there exists a homomorphism such that
\[
 f(x)=f(1)\varphi(x)
 \qquad
 \left(x\in \mathbb{K}\right),
\]
moreover we also have
\[
 \sum_{i=1}^{n}a_{i}^{q_{i}}f(1)^{q_{i}}=0.
\]
\end{rem}

As a consequence of the previous statement and  Theorem
\ref{thm_indep2} we have the following.

\begin{thm}\label{thm_sep}
Let $\mathbb{K}\subset \C$ be a field of finite transcendence degree over $\mathbb{Q}$. 
Each additive solution of equation \eqref{Eq1.3} must be of the form
\[
f_i= \sum_{j=1}^{n-1}P_{i,j}\varphi_{j},
\]
where $P_{i,j}$'s are  polynomials on
$\mathbb{K}^{\times}$ and $\varphi_j\colon\mathbb{K}\to \mathbb{C}$
are field homomorphisms.
 Then
 \[
 \tilde{f}_{i}(x)=P_{i,j}\varphi_{j}(x)
 \qquad
 \left(x\in \mathbb{K}, j=1, \ldots, n-1\right).
\]
 is also a solution of \eqref{Eq1.3} and all $\tilde{f}_{i}$  are additive.
\end{thm}

\begin{proof}
By Theorem \ref{Thm_main}, the solutions $f_i\colon \mathbb{K}\to
\mathbb{C}$ of \eqref{Eq1.3} are almost exponential  polynomials of
the Abelian group $\mathbb{K}^{\times}$. Since $\mathbb{K}$ is a field of finite transcendence degree, 
by Remark \ref{rem_diff} and Theorem \ref{Szekely} all $f_i$'s are exponential polynomials. 
Thus there are nonnegative integers $k,l\le n-1$ and distinct
(nonconstant) exponential functions $m_1, \dots, m_k\colon
\mathbb{K}^{\times}\to \mathbb{C}$, further additive functions $a_1,
\dots,a_l\colon G\to \mathbb{C}$ that are linearly independent over
$\mathbb{C}$   and classical complex polynomials $P_{i,1}, \dots,
P_{i,k}\colon \mathbb{C}^l\to \mathbb{C}$ with $\deg{P_{i,j}}\le
n-1$  be  such that
\begin{equation}\label{eqext}
    f_i=\sum_{j=1}^k P_{i,j}(a_1,\dots, a_l)m_j.
\end{equation}
Substituting $f_i$ to \eqref{Eq1.3}, we have
\begin{equation}
 0=\sum_{i=1}^{n}f^{q_{i}}_{i}\left(x^{p_{i}}\right)
 =\sum_{i=1}^{n}\left( \sum_{j=1}^{k} P_{i,j}(a_1,\dots, a_l)m_j \right)^{q_i}(x^{p_i}).
\end{equation}

Since $m_{1}, \ldots, m_{k}$ are distinct (nonconstant)
exponentials, the coefficients of the terms $m_1^{a_1}\cdots
m_k^{a_k}$ in the expansion must be $0$. Taking all terms that
contains only $m_j$ as an exponential in the product. By this
reduction, we get that
\begin{equation}\label{eqet}
    \sum_{i=1}^n (P_{i,j} m_j)^{q_i}(x^{p_i})=0
\end{equation}
holds for all $j=1, \dots, k$. 

The additive functions with respect to addition on $\mathbb{K}$
constitute a linear space that is translation invariant with respect
to multiplication on $\mathbb{K}^{\times}$. By Lemma \ref{lem_tils},
we get that if $\sum_{j=1}^k (P_{i,j} m_j)$ is additive (with
respect to addition on $\mathbb{K}$), then $P_{i,j}\cdot m_j$ and
$m_j$ are additive for every $j=1, \dots,k$. The first implies that
$\tilde{f}_i$ is additive. Since $m_j$ is additive on $\mathbb{K}$
that has finite transcendence degree and multiplicative on
$\mathbb{K}^{\times}$, by Lemma \ref{lem_hom} $m_j$ can be extended
as an automorphism $\phi_j$ of $\mathbb{C}$. These imply the
statement.

\end{proof}

\begin{rem}\label{rem_noimp}
It is worth to note that the role of homomorphism $m$ lost its
importance. By Theorem \ref{thm_sep} for finding a solution of
\eqref{Eq1.3} it is enough to find all solutions of \eqref{eqet}
separately for every $j=1, \dots, k$.
Since $N=p_1q_1=\cdots =p_nq_n$ and $m_j\neq 0$, equation \eqref{eqet} is equivalent to 
\begin{equation}\label{eqegs}\sum_{i=1}^n P^{q_i}_{i,j}(x^{p_i})=0.\end{equation}
Conversely, if \eqref{eqegs} holds and  $P_{i,j}(x)\cdot x$ is
additive, then $f_i=P_{i,j}\varphi$ is an additive solution of
\eqref{eqet}, where $\varphi$ is an arbitrary homomorphism.
\end{rem}

\begin{rem}\label{rem_noadd}
Our next aim is to prove Theorem \ref{thm_auto}. If we omit the
condition of additivity of $f_i$ then we can easily find solutions
that are neither homomorphisms, nor differential operators as it can
be seen  in Example \ref{ex3}.
\end{rem}
\begin{ex}\label{ex3}
To illustrate this, let us consider the following equation on a
field $\mathbb{K}$.
\begin{equation}\label{eqreex1}
    f(x^4)+g^2(x^2)+h^4(x)=0
\qquad \left(x\in \mathbb{K}\right),
\end{equation}
where $f, g, h \colon \mathbb{K}\to \mathbb{K}$ denote the unknown
(not necessarily additive) functions. 

Let $d\colon \mathbb{K}\to \mathbb{K}$ be a nontrivial derivation and define the 
function $a\colon \mathbb{K}^{\times} \to \mathbb{K}$ by 
\[
 a(x)= \dfrac{d(x)}{x} 
 \qquad 
 \left(x\in \mathbb{K}^{\times}\right). 
\]
Then $a$ is an additive function on group $\mathbb{K}^{\times}$. 

Consider the functions $f, g$ and $h$ defined through
\begin{align*}
    f(x)&=-(20+4a(x)+a^2(x))x,\\
    g(x)&=2(1+a(x))x,\\
    h(x)&=2x, 
\end{align*}
that clearly provide a solution for \eqref{eqreex1}. 
Indeed, using $a^k(x^l)=l^k\cdot a^k(x)$ for all $l,k\in \mathbb{N}$ we have
\begin{align*}
    f(x^4)&=-(20+4a(x^4)+a^2(x^4))x^4=-(20+16a(x)+16a^2(x))x^4,\\
    g^2(x^2)&=(2(1+a(x^2)))^2 x^4=(4+16a(x)+16a^2(x))x^4,\\
    h^4(x)&=16x^4.
\end{align*}

On the other hand,  it does not satisfies \eqref{Eq_dep}. Clearly,
$f(1)=-20, g(1)=2, h(1)=2$ and
$$f(x)+g(1)g(x)+h^3(1)h(1)=(-20-a(x)-a^2(x)+4+4a(x)+8)x=-a^2(x)+3a(x)-8\ne
0.$$ This is caused by the fact that at least one of the function $f,g$ and $h$ is not
additive. 
It is easy to check that $g$ and $h$ are additive on $\mathbb{K}$, but $f$ is not.
\end{ex}

\begin{thm}\label{thm_auto}
 Let $n\in \mathbb{N}$ be arbitrary, $\mathbb{K}$ a field,
 $f_{1}, \ldots, f_{n}\colon \mathbb{K}\to \mathbb{C}$ additive functions. Suppose further that we are given natural numbers
 $p_{1}, \ldots, p_{n}, q_{1}, \ldots, q_{n}$ such that they fulfill condition $(\mathscr{C})$.
 If
 \begin{equation}\label{eqfi}
  \sum_{i=1}^{n}f_{i}^{q_{i}}\left(x^{p_{i}}\right)=0
 \end{equation}
holds for all $x\in \mathbb{K}$, then there exist homomorphisms
$\varphi_{1}, \ldots, \varphi_{n-1}\colon \mathbb{K}\to \mathbb{C}$
and $\alpha_{i, j}\in\mathbb{C}\, (i=1, \ldots, n; j=1, \ldots,
n-1)$ so that
\begin{equation}\label{eq_lincomb}
 f_{i}(x)=\sum_{j=1}^{n-1}\alpha_{i, j}\varphi_{j}(x)
 \qquad
 \left(x\in \mathbb{K}\right).
\end{equation}
Moreover  $\alpha_{i,j}\varphi_j$ gives also a solution of
\eqref{eqfi}.
\end{thm}

\begin{proof}
Let us assume first that $\mathbb{K}\subset\C$ be a field of finite
transcendence degree over $\mathbb{Q}$. By Theorem \ref{thm_sep} we
can restrict our attention on the solutions $f_i=P_i\cdot \varphi$.
Namely,
\[
0=\sum_{i=1}^n(P_i\cdot\varphi)(x^{p_i})^{q_i}=\varphi(x^N)\cdot
\sum_{i=1}^n P_i^{p_i}(x^{p_i})^{q_i} \qquad \left(x\in
\mathbb{K}\right).
\]

Clearly $\varphi$ has no special role in the previous equation (see
Remark \ref{rem_noimp}), thus $\varphi\equiv id$ can be assumed
along the proof. Therefore the solutions are $f_i=P_i\cdot id$.
By Theorem \ref{thm_deriv} we can identify $f_i=P_i\cdot x$ with a
derivation $D_i$ defined with \eqref{eqdiff}, where the degree of
$P_i$ is the same as the order of $D_i$. Let us denote the maximal
degree of all $P_i$ by $M$. Note that $D_{i}$ can be uniquely
written in terms of the elements of the basis $\mathscr{B}$ defined
as in Remark \ref{rem_diff}.

Let the elements of $\mathscr{B}$ be the functions $x, d_{1}, \dots, d_k, \dots, d_{i_1}\circ\cdots\circ d_{i_s}(x)$ that are linearly independent over $\C$ for all $i_1,\dots, i_s<n$. 
Since every composition is an additive function on $\mathbb{K}$, by
Theorem \ref{thm_indep} we get that the elements of $\mathscr{B}$
are also algebraically independent.

Now fix $i$ such that $D_i$ has maximal order $M$ and $q_i$ is the
smallest possible.  Thus it contains a term $d_{j_1}\circ\cdots\circ
d_{j_M}\in \mathscr{B}$.  Then we have that
$$d_{j_1}\circ\cdots\circ d_{j_M}(x^{p_i})=p_ix^{p_i-1}(d_{j_1}\circ\dots \circ d_{j_M})(x)+p_i(p_i-1)x^{p_i-2}d_1(x) (d_2\circ \dots \circ d_{j_M})(x)+\dots$$

Let us assume that $M>1$. Since $x, d_{j_1}\circ\cdots\circ
d_{j_M}(x)\in \mathscr{B}$ and they are distinct, the coefficient of
\begin{equation}\label{eq_depo}x^{q_i(p_i-1)}(d_{j_1}\circ\dots
\circ d_{j_M}(x))^{q_i}\end{equation} uniquely determined and it
must vanish.
In $D_i(x^{p_i})^{q_i}$ we have only the term of \eqref{eq_depo} with nonzero coefficient. 
Since $q_i$ was minimal,  $D_{j}(x^{p_{j}})^{q_{j}} $ does not
contain the product \eqref{eq_depo}, if $j \ne i$. In such a
situation however this term cannot vanish, contradicting to the algebraic independence. 
This leads to the fact that $M=1$, i.e. every
$D_i(x)=c_i \cdot x$, for some complex constant $c_i$.

This clearly implies in general that every solution can be written
as $$f_i(x)=\sum_{j=1}^{n-1}c_{i,j} \varphi_j(x),$$ for some
constants $c_{i,j}\in \mathbb{C}$ and field homomorphisms
$\varphi_1, \dots \varphi_{n-1}\colon \mathbb{K}\to \mathbb{C}$.

Now let $\mathbb{K}$ be an arbitrary field of characteristic 0 and
assume that the statement is not true. Then by Theorem
\ref{Thm_main} there exist almost exponential polynomial solutions
defined on $\mathbb{K}^{\times}$ such that
\[
f_i =\sum_{j=1}^{n-1}P_{i,j}\varphi_{j} \not\equiv
\sum_{j=1}^{n-1}\alpha_{i, j}\varphi_{j}.
\]
Then there exists a finite set $S\subset \mathbb{K}$ which
guarantees this. The field generated by $S$ over $\mathbb{Q}$ is
isomorphic to field $\mathbb{K}\subset\C$ of finite transcendence
degree. Let us denote this isomorphism by $\Phi:\mathbb{Q}(S)\to
\mathbb{K}$. The previous argument provides that $f_i\circ \Phi$
satisfy \eqref{eq_lincomb}. Since $\Phi^{-1}$ is also an
isomorphism, $f_i$ satisfies \eqref{eq_lincomb}, as well. This
contradicts our assumption and finishes the proof.
\end{proof}
 \begin{rem}
 Here we note that the proof of Theorem \ref{thm_auto} essentially uses the fact that the field $\mathbb{K}$ has characteristic 0, that we assume throughout the whole paper.
 \end{rem}

The following example illustrates a special case when not all of
$f_i$ are of the form $c\cdot \varphi$. Theorem \ref{thmph} is
devoted to show that this is in some sense the exceptional case.
\begin{ex}\label{ex1}
  Let $\mathbb{K}$ be a field and $f, g, h\colon \mathbb{K}\to \mathbb{C}$ be additive functions such that
 \[
  f(x^{4})+g^{2}(x^{2})+h^{4}(x)=0
 \]
holds for all $x\in \mathbb{K}$. According to Theorem \ref{Thm_main}
define the $4$-additive function $F\colon \mathbb{K}^{4}\to
\mathbb{C}$ through
\begin{multline*}
 F(x_{1}, x_{2}, x_{3}, x_{4})=
 f(x_{1}x_{2}x_{3}x_{4})
 +\frac{1}{3}\left\{
 g(x_{1}x_{2})g(x_{3}x_{4})+g(x_{1}x_{3})g(x_{2}x_{4})\right.
 \\
 \left. +g(x_{1}x_{4})g(x_{2}x_{3})\right\}
 +h(x_{1})h(x_{2})h(x_{3})h(x_{4})
 \qquad
 \left(x\in \mathbb{K}\right).
\end{multline*}
The above equation yields that the trace of $F$ is identically zero,
thus $F$ itself is identically zero, too. From this we immediately
get that
\[
 F(x, 1, 1, 1)= h^3\left(1\right) h\left(x\right)+g\left(1\right) g\left(x\right)
 +f\left(x\right)
 \qquad
 \left(x\in \mathbb{K}\right),
\]
that is, the functions $f, g, h$ are linearly dependent. Using this,
we also have that
\[
 0
 =F(x, y, 1, 1)
 =-3 h^3\left(1\right) h\left(x y\right)-2 g\left(1\right) g
 \left(x y\right)+3 h^2\left(1\right) h\left(x\right) h\left(y
 \right)+2 g\left(x\right) g\left(y\right)
\]
has to be fulfilled by any $x, y\in \mathbb{K}$.

Define the functions $\chi, \varphi_{1}, \varphi_{2}\colon
\mathbb{K}\to \mathbb{C}$ as
\[
\begin{array}{rcl}
 \chi(x)&=& 3h(1)^{3}h(x)+2g(1)g(x)\\
 \varphi_{1}(x)&=&\sqrt{3}h(1)h(x)\\
 \varphi_{2}(x)&=&\sqrt{2}g(x)
 \end{array}
 \qquad
 \left(x\in \mathbb{K}\right)
\]
to obtain the Levi-Civit\`{a} equation
\[
 \chi(xy)=\varphi_{1}(x)\varphi_{1}(y)+\varphi_{2}(x)\varphi_{2}(y)
 \qquad
 \left(x, y\in \mathbb{K}\right).
\]

Using Theorems \ref{Szekely} and \ref{Thm_main}, we deduce that
there are homomorphisms $\varphi_{1}, \varphi_{2}\colon
\mathbb{K}\to \mathbb{C}$ and complex constants $\alpha_{1},
\alpha_{2}, \beta_{1}, \beta_{2}, \gamma_{1}, \gamma_{2}$ so that
\[
\begin{array}{rcl}
 g(x)&=&\alpha_{1}\varphi_{1}(x)+\alpha_{2}\varphi_{2}(x)\\
 h(x)&=&\beta_{1}\varphi_{1}(x)+\beta_{2}\varphi_{2}(x)\\
 f(x)&=&\gamma_{1}\varphi_{1}(x)+\gamma_{2}\varphi_{2}(x)
 \end{array}
 \qquad
 \left(x\in \mathbb{K}\right),
 \]
where the above complex numbers will be determined from the
functional equation.

Indeed, from one hand we have
\begin{multline*}
 -f(x^{4})= g^{2}(x^{2})+h^{4}(x)
 =
 \left(\alpha_{1}\varphi_{1}(x^{2})+\alpha_{2}\varphi_{2}(x^{2})\right)^{2}+
 \left(\beta_{1}\varphi_{1}(x)+\beta_{2}\varphi_{2}(x)\right)^{4}
 \\
 =
 \alpha_{1}\varphi_{1}(x)^{4}+2\alpha_{1}\alpha_{2}\varphi_{1}(x)^{2}\varphi_{2}(x)^{2}+\alpha_{2}^{2}\varphi_{2}(x)^{4}
 \\
  +\beta_{1}^{4}\varphi_{1}(x)^{4}+4\beta_{1}^{3}\beta_{2}\varphi_{1}(x)^{3}\varphi_{2}(x)+6\beta_{1}^{2}\beta_{2}^{2}\varphi_{1}(x)^{2}\varphi_{2}(x)^{2}+
 4\beta_{1}\beta_{2}^{3}\varphi_{1}(x)\varphi_{2}(x)^{3}+\beta_{2}^{4}\varphi_{2}(x)^{4} \\
 =
 \left(\alpha_{1}^{2}+\beta_{1}^{4}\right)\varphi_{1}(x)^{4}
 +
 \left(2\alpha_{1}\alpha_{2}+6\beta_{1}^{2}\beta_{2}^{2}\right)\varphi_{1}(x)^{2}\varphi_{2}(x)^{2}
 +
 \left(\alpha_{2}^{2}+\beta_{2}^{4}\right)\varphi_{2}(x)^{4}
 \\
 +4\beta_{1}^{3}\beta_{2}\varphi_{1}(x)^{3}\varphi_{2}(x)+ 4\beta_{1}\beta_{2}^{3}\varphi_{1}(x)\varphi_{2}(x)^{3}
 \end{multline*}
for all $x\in \mathbb{K}$.

On the other hand
\[
 -f(x^{4})=-\gamma_{1}\varphi_{1}(x^{4})-\gamma_{2}\varphi_{2}(x^{4})
 =
 -\gamma_{1}\varphi_{1}(x)^{4}-\gamma_{2}\varphi_{2}(x)^{4}
 \qquad
 \left(x\in \mathbb{K}\right).
\]
Bearing in mind Theorem \ref{L_lin_dep}, after comparing the
coefficients, we have especially that equations
\[
 \begin{array}{rcl}
  \alpha_{1}^{2}+\beta_{1}^{4}&=&-\gamma_{1}\\
  \alpha_{2}^{2}+\beta_{2}^{4}&=&-\gamma_{2}\\
  \alpha_{1}\alpha_{2}&=&0\\
  \beta_{1}\beta_{2}&=&0
 \end{array}
\]
have to be fulfilled. This yields however that
\[
 \begin{array}{rcl}
  f(x)&=&-g(1)^{2}\varphi_{1}(x)-h^{4}(1)\varphi_{2}(x)\\
  g(x)&=&g(1)\varphi_{1}(x)\\
  h(x)&=&h(1)\varphi_{2}(x)
 \end{array}
\qquad
 \left(x\in \mathbb{K}\right).
\]

\end{ex}

\

\

Without the loss a generality we can (and we also do) assume that
the parameters $q_{1}, \ldots, q_{n}$ are arranged in a strictly
increasing order, that is, $q_1<q_2<\dots<q_n$ holds.

\begin{thm}\label{thmph}
 Let $n\in \mathbb{N}$ be arbitrary and $\mathbb{K}$ a field.
 Assume that there are given natural numbers
 $p_{i}, q_{i}$ $(i=1, \dots,n)$ so that condition $(\mathscr{C})$ is satisfied.
 Let $f_1, \ldots, f_n$ be additive solutions of \begin{equation}\label{eq_vege}
     \sum_{i=1}^nf_i^{q_i}(x^{p_i})=0
 \end{equation}
 Then
 \begin{equation}
     f_i=
     \begin{cases}
     c_{i,j}\varphi_j& \textrm{ if } i>1 \textrm{ or } q_1\ne 1,
     \\
     \displaystyle\sum_{j=1}^{n-1} c_{1,j}\varphi_j & \textrm{ if } i=1 \textrm{ and } q_1=1,
  \end{cases}
  \end{equation}
where $\ph_1,\dots \ph_{n-1}\colon \mathbb{K}\to \mathbb{C}$ are
arbitrary field homomorphisms and $\displaystyle\sum_{i=1}^{n-1}
c_{i,j}^{q_i}=0$ for all $j=1, \dots, n$.
\end{thm}

\begin{proof}
By Theorem \ref{thm_auto} every solution \[
 f_{i}(x)=\sum_{j=1}^k c_{i,j}\ph_{j}(x)
 \qquad
 \left(x\in \mathbb{K}\right),
\] for some $c_{i,j}\in \mathbb{C}$
 thus the statement for $f_1$ if $q_1=1$ is trivial.

We show the rest of the statement by using a descending process as
follows.

Introducing the formal variables $x_1=\ph_1(x), \ldots,
x_k=\ph_k(x)$, equation \eqref{eqfi} yields that
\begin{equation}\label{eqex1}
\sum_{i=1}^n\left(c_{i,1}x_1^{p_i}+\ldots+c_{i,k}x_k^{p_i}\right)^{q_i}=0.
\end{equation}
By the polynomial theorem
\begin{equation}\label{eqex2}
\sum_{i=1}^n\sum_{J_{i,1}+\ldots+J_{i,k}=q_i}\frac{q_i!}{J_{i,1}!\cdot
\ldots \cdot J_{i,k}!}c_{i,1}^{J_{i,1}}\cdot \ldots \cdot
c_{i,k}^{J_{i,k}}\cdot x_1^{J_{i,1}p_i}\cdot \ldots \cdot
x_k^{J_{i,k}p_i}=0.
\end{equation}
Since we have distinct homomorphisms it follows, by Theorem \ref{thm_indep2}, that
the coefficient of each monomial term of the polynomial in equation
\eqref{eqex2} must be zero. Two addends belong to the same monomial
term if and only if
\[
J_{i,1}p_{i}=J_{j,1}p_{j}, \ldots, J_{i,k}p_{i}=J_{j,k}p_{j}.
\]
If $q_i\geq 2$ then we choose the values $J_{i,1}=1$,
$J_{i,2}=q_i-1, \ J_{i,3}=\ldots=J_{i,k}=0.$ For each addend
belonging to the same monomial term
\begin{align*}
      p_i & =J_{j,1}p_j, \\
     p_i(q_i-1) & =J_{j,2}p_j,\\
     J_{j,3}  =\ldots & =J_{j,k}=0.
\end{align*}

This means that $p_j$ divides $p_i$ or, in an equivalent way, $q_i$
divides $q_j$. Without loss of generality we can suppose that $q_i$
is the maximal among the possible powers. Therefore there is no any
addend belonging to the same monomial term as
$x_1^{p_i}x_2^{p_i(q_i-1)}$. Since $q_i\geq 2$ it follows that
$c_{i,1}=0$ or $c_{i,2}=0$. Repeating the argument for arbitrary
pair $x_k$ and $x_l$ we get that except at most one  $c_{i,j}=0$.
This immediately implies equation \eqref{eq_vege}.

Finally, condition
\[
\sum_{i=1}^n c_{i,j}^{q_i}=0
\]
clearly follows from Theorem \ref{thm_sep}.
\end{proof}

\begin{ex}\label{ex2}
 Let $\mathbb{K}$ be a field. Illustrating the previous results we consider all additive solutions $f_{1}, f_{2}, f_{3}, f_{4}\colon \mathbb{K}\to \mathbb{C}$ of
 \begin{equation}
     f_{1}^{2}(x^{6})+f_{2}^{3}(x^{4})+f_{3}^{4}(x^{3})+f_{4}^{6}(x^{2})=0
  \qquad
  \left(x\in \mathbb{K}\right).
 \end{equation}
with   $f_{i}(1)\neq 0$ for $i=1, 2, 3, 4$.

We distinguish two cases. If every $f_i$ is of the form
$c_i\varphi$, then
\[
c_1^2+c_2^3+c_3^4+c_4^6=0
\]
and $\varphi$ can be any homomorphism.

If not, then there are two different field homomorphisms
$\varphi_{1}, \varphi_{2}$ such that
\begin{align*}
f_{i}&=c_{i}\varphi_1,\\
f_{j}&=c_{j}\varphi_2.
\end{align*}
for some $1\le i\ne j\le 4$.

Practically, the only possible option is that $i_1, i_2 \in \{1,
2,3, 4\}$ are such that
\begin{align*}
f_{i_1}&=c_{i_1}\varphi_1,\\  f_{i_2}&=c_{i_2}\varphi_1.
\end{align*}
and for $j_1, j_2\in \{1, 2,3, 4\}\setminus \{i_1, i_2\}$ we have
\begin{align*}
f_{j_1}&=c_{j_1}\varphi_2,\\  f_{j_2}&=c_{j_2}\varphi_2.
\end{align*}
It also clearly follows that
\[
\begin{array}{rcl}
c_{i_1}^{q_{i_1}}+c_{i_2}^{q_{i_2}}&=&0\\[0.4cm]
c_{j_1}^{q_{j_1}}+c_{j_2}^{q_{j_2}}&=&0.
\end{array}
\]

For instance, if $i_1=1, i_2=2, j_1=3, j_4=4$, then we get that
\[
\begin{array}{rcl}
f_1&=&c_1 \varphi_1,\\
f_2&=&c_2 \varphi_1, \\
f_3&=&c_3 \varphi_2, \\
f_4&=&c_4 \varphi_4,
\end{array}
\]
where $c_1^2+c_2^3=c_3^4+c_4^6=0$ and
$\varphi_1,\varphi_2\colon\mathbb{K}\to \mathbb{C}$ are arbitrary
field homomorphisms. \
\end{ex}

\subsection{Summary}

We can assume that $0<q_1<q_2<\dots<q_n$. As a consequence of
Theorem \ref{thmph} we get that for a given system of solutions
$f_i$ of \eqref{eqfi} the index set $\mathscr{I}=\{1,\dots, n\}$ can
be decomposed into some subsets $\mathscr{I}_1, \dots,
\mathscr{I}_k$ ($k<n$) such that
\[
\bigcup_{j=1}^k \mathscr{I}_j=\mathscr{I}
\]
and
\[
\bigcap_{j=1}^k \mathscr{I}_j =
\begin{cases}
\emptyset& \textrm{ if } q_1\ne 1\\
\{1\} & \textrm{ if } q_1= 1.
\end{cases}
\]
If $q_1\ne 1$, then for every $\mathscr{I}_j$ ($j=1, \dots, k<n$)
there exists injective homomorphisms $\varphi_j\colon\mathbb{K}\to
\mathbb{C}$ such that $f_i=c_i\varphi_j$ and $\sum_{i\in
\mathscr{I}_j}c_i^{q_i}=0$. If $q_1=1$, then $f_1=\sum_{j=1}^k
c_{1,j} \varphi_j$, $f_i=c_i\varphi_j$ if $1\ne i \in \mathscr{I}_j$
and $$c_{1,j}+\sum_{i\in \mathscr{I}_j, i\ne 1} c_i^{q_i}=0.$$

Conversely, if there are given a partition $\mathscr{I}_{j}$
($j=1,\dots,k$) of $\{1, \dots, n\}$ such that except maybe element
$1$, the sets are disjoint, then for every field homomorphism
$\varphi_1, \dots, \varphi_k\colon\mathbb{K}\to \mathbb{C}$  we get
a solution of \eqref{eqfi} as
\[
f_i =
\begin{cases}
c_i\varphi_j & \textrm{ if } i\in \mathscr{I}_j \textrm{ and either } i\ne 1 \textrm{ or } q_1\ne 1 \\
\sum_{j=1}^k c_{1,j} \varphi_j & \textrm{ if } q_1= 1 \textrm { and
} i=1,
\end{cases}
\]
where $\sum_{i\in +\mathscr{I}_j} c_i^{q_i}=0$ if $q_1\ne 1$,
otherwise $c_{1,j}+\sum_{i\in +\mathscr{I}_j, i\ne 1} c_i^{q_i}=0$.
Additionally, we get that for every set $\mathscr{I}_j$ the system
of $f_i=c_i\varphi_j$ ($i\in \mathscr{I}_j$), where $c_i$ satisfy
the previous equation, is a solution of \eqref{eqfi}.
This is a sub-term of \eqref{eqfi}, thus it seems reasonable that we
are just looking for solutions that do not satisfies any partial
equation of \eqref{eqfi}.

We say that the system of functions $f_{1}, \ldots, f_{n}$ form an
\emph{irreducible} solution if it does not satisfy a sub-term of
\eqref{eqfi}.

\begin{cor}\label{thm1}
 Under the assumptions of Theorem \ref{thmph}, let $f_{1}, \ldots, f_{n}\colon \mathbb{K}\to \mathbb{C}$
 be additive \emph{irreducible} solutions of \eqref{eqfi}.
 Then for all $i= 1, \ldots, n$,
 \[
 f_i(x)=c_i \cdot \ph(x) \qquad (x\in \K),
 \]
 where $\ph\colon \mathbb{K}\to \mathbb{C}$ is an arbitrary field homomorphism and $c_i\in \mathbb{C}$ satisfies
\[
\sum_{i=1}^n c_i^{q_i}=0.
\]
\end{cor}

\subsection{Special cases}
The following statement which is only about the \emph{real-valued}
solutions, is an easy observation which allows us to focus on the
important cases henceforth.

\begin{prop}
  Let $n\in \mathbb{N}$ be arbitrary, $\mathbb{K}$ a field,
 $f_{1}, \ldots, f_{n}\colon \mathbb{K}\to \mathbb{R}$ additive functions. Suppose further that we are given natural numbers
 $p_{1}, \ldots, p_{n}, q_{1}, \ldots, q_{n}$ so that they fulfill condition $(\mathscr{C})$.
 If equation  \eqref{Eq1.3} is satisfied for all $x\in \mathbb{K}$ by the functions $f_{1}, \ldots, f_{n}$ and
the parameters fulfill
\[
 q_{i}=2k_{i}
 \qquad
 \left(i=1, \ldots, n\right)
\]
with certain positive integers $k_{1}, \ldots, k_{n}$, then all the
functions $f_{1}, \ldots, f_{n}$ are identically zero.
\end{prop}

\begin{proof}
 If the parameters fulfill
\[
 q_{i}=2k_{i}
 \qquad
 \left(i=1, \ldots, n\right)
\]
with certain positive integers $k_{1}, \ldots, k_{n}$, then equation
\eqref{Eq1.3} can be rewritten as
\[
 \sum_{i=1}^{n}\left(f^{k_{i}}_{i}\left(x^{p_{i}}\right)\right)^{2}=0
 \qquad
 \left(x\in \mathbb{K}\right),
\]
in other words, we received that the sum of nonnegative real numbers
has to be zero, that implies that all the summands has to be zero
for all $x\in \mathbb{K}$. Thus the functions $f_{1}, \ldots,
f_{n}\colon \mathbb{K}\to \mathbb{R}$ are identically zero.
\end{proof}

As an application of the results above, first we study the case
\[
f_i(x)=a_i\cdot f(x) \qquad \left( x\in \mathbb{K},\; i=1, \ldots, n
\right),
\]
where $a_{1}, \ldots, a_{n}$ are given complex numbers so that at
least one of them is nonzero.

\begin{thm}\label{thm4}
 Let $n\in \mathbb{N}$ be arbitrary, $\mathbb{K}$ a field. Assume that there are given natural numbers
 $p_{1}, \ldots, p_{n}, q_{1}, \ldots, q_{n}$ so that they fulfill condition $(\mathscr{C})$.
 The function
 $ f\colon \mathbb{K}\to \mathbb{C}$ is an additive solution of
 \begin{equation}\label{Eq1.62}
  \sum_{i=1}^{n}(a_i\cdot f)^{q_{i}}\left(x^{p_{i}}\right)=0
 \end{equation}
 if and only if
 \begin{equation}\label{eqfe1}
     f(x)=c \cdot \ph(x),
      \end{equation}
where $\ph\colon \mathbb{K}\to \mathbb{C}$ is a homomorphism and for
the constant $c$ equation
\begin{equation}\label{eqfe2}
    \sum_{i=1}^n (c\cdot a_i)^{q_i}=0
\end{equation}
also has to be satisfied.
\end{thm}

According to a result of Darboux \cite{Dar77}, the only function
$f\colon \mathbb{R}\to \mathbb{R}$ that is additive and
multiplicative is of the form
\[
 f(x)=0
 \qquad
 \text{or}
 \qquad
 f(x)=x
 \qquad
 \left(x\in \mathbb{R}\right).
\]
From this, we get also that every homomorphism $f\colon
\mathbb{R}\to \mathbb{C}$ is of the form
\[
 f(x)=\kappa\cdot x
 \qquad
 \left(x\in \mathbb{R}\right),
\]
where $\kappa\in \left\{ 0, 1\right\}$.

\begin{cor}
  Let $n\in \mathbb{N}$ be arbitrary and assume that there are given natural numbers
 $p_{1}, \ldots, p_{n}, \allowbreak q_{1}, \ldots, q_{n}$ so that they fulfill condition $(\mathscr{C})$.
 Let
 $f_{1}, \ldots, f_{n}\colon \mathbb{R}\to \mathbb{C}$ be additive solutions of \eqref{Eq1.3}.
 Then and only then, there are complex numbers $c_{1}, \ldots, c_{n}$ with the property
 \[
  \sum_{i=1}^{n}c_{i}^{q_{i}}=0
 \]
so that for all $i=1, \ldots, n$
\[
 f_{i}(x)=c_{i}\cdot x
 \qquad
 \left(x\in \mathbb{R}\right).
\]

\end{cor}

The above corollary shows that for real functions every solution of
equation \eqref{Eq1.3} is automatically continuous (in fact even
analytic) without any regularity assumption.

\section{Open problems and perspectives}\label{s5}

In the last section of our paper we list some open problems as well
as we try to open up new perspectives concerning the investigated
problem.

\begin{dfn}
Let $(G, +)$ be an Abelian group and $n\in \mathbb{N}$, a function
$f\colon G\to \mathbb{C}$ is termed to be a \emph{(generalized)
monomial of degree $n$} if it fulfills the so-called \emph{monomial
equation}, that is,
\[
\Delta^{n}_{y}f(x)=n! f(y) \qquad \left(x, y\in G\right).
\]
\end{dfn}

\begin{rem}
 Obviously generalized monomials of degree $1$ are nothing else but additive functions.
 Furthermore, generalized monomials of degree $2$ are solutions of the equation
 \[
\Delta^{2}_{y}f(x)=n! f(y) \qquad \left(x, y\in G\right),
\]
which is equivalent to the so-called \emph{square norm equation},
i.e.,
\[
 f(x+y)+f(x-y)=2f(x)+2f(y)
 \qquad
 \left(x, y\in G\right).
\]
In this case for the mapping $f\colon G\to \mathbb{C}$ the term
\emph{quadratic mapping} is used as well.
\end{rem}

\begin{prop}
 Let $G$ be an Abelian group and $n\in \mathbb{N}$. A function $f\colon G\to \mathbb{C}$
 is a generalized monomial of degree $n$, if and only if, there exists a symmetric, $n$-additive
 function $F\colon G^{n}\to \mathbb{C}$ so that
 \[
  f(x)=F(x, \ldots, x)
  \qquad
  \left(x\in G\right).
 \]
\end{prop}

\begin{opp}[Higher order generalized monomial solutions]
 In this paper we determined the additive solutions of equation \eqref{Eq1.3}.
 It would be however interesting to determine the higher order monomial solutions of the equation in question.
 More precisely, the following problem would also be of interest.
 Let $n, k\in \mathbb{N}$ be arbitrary, $\mathbb{K}$ a field, $f_{1},
\ldots, f_{n}\colon \mathbb{K}\to \mathbb{C}$  generalized monomials
of degree $k$. Suppose further that we are given natural numbers
$p_{1}, \ldots, p_{n}, q_{1}, \ldots, q_{n}$ so that
\[
\tag{$\mathscr{C}$}
 \begin{array}{rcl}
  p_{i}&\neq&p_{j} \qquad \text{ for } \quad i\neq j\\
  q_{i}&\neq&q_{j} \qquad \text{ for } \quad i\neq j\\
  1<p_{i}\cdot q_{i}&=&N\qquad \text{ for } i=1, \ldots, n
 \end{array}
\]
Suppose also that equation
\begin{equation}
 \sum_{i=1}^{n}f^{q_{i}}_{i}\left(x^{p_{i}}\right)=0
\end{equation}
is satisfied. Here the question is, whether we can say something
more about these functions $f_{1}, \ldots\allowbreak , f_{n}$?

We remark that in case $k\geq 2$, we do not know whether such `nice'
representation for the functions $f_{1}, \ldots, f_{n}$ as in
Theorem \ref{Thm_main} can be expected.

At the same time, there are cases when the representation is `nice'
as well as previously. To see this, let us consider the following
problem. Assume that for the quadratic function $f\colon
\mathbb{K}\to \mathbb{C}$ we have
\[
 f(x^{2})=f(x)^{2}
 \qquad
 \left(x\in \mathbb{K}\right).
\]
Since $f$ is a generalized monomial of degree $2$, there exists a
symmetric bi-additive function $F\colon \mathbb{K}^{2}\allowbreak
\to \mathbb{C}$ so that
\[
 F(x, x)= f(x)
 \qquad
 \left(x\in \mathbb{K}\right).
\]
Define the symmetric $4$-additive mapping $\mathcal{F}\colon
\mathbb{K}^{4}\to \mathbb{C}$ through
\begin{multline*}
 \mathcal{F}(x_{1}, x_{2}, x_{3}, x_{4})=
 F(x_{1}x_{2}, x_{3}x_{4})
 +  F(x_{1}x_{3}, x_{2}x_{4})
 + F(x_{1}x_{4}, x_{2}x_{3})
 \\
 -F(x_{1}, x_{2})F(x_{3}, x_{4})
 -F(x_{1}, x_{3})F(x_{2}, x_{4})
 -F(x_{1}, x_{4})F(x_{2}, x_{3})
 \qquad
 \left(x_{1}, x_{2}, x_{3}, x_{4}\in \mathbb{K}\right).
\end{multline*}
Since
\[
 \mathcal{F}(x, x, x, x)= 3 \left(F(x^{2}, x^{2})-F(x, x)^{2}\right)=3 \left(f(x^{2})-f(x)^{2}\right)=0
 \qquad
 \left(x\in \mathbb{K}\right),
\]
the mapping $\mathcal{F}$ has to be identically zero on
$\mathbb{K}^{4}$. Therefore, especially
\[
0=\mathcal{F}(1, 1, 1, 1)= 3F(1, 1)-3F(1, 1)^{2},
\]
yielding that either $F(1, 1)=0$ or $F(1, 1)=1$. Moreover,
\[
 0= \mathcal{F}(x, 1, 1, 1) =
 3F\left(x , 1\right)-3F\left(1 , 1\right)F\left(x , 1\right)
 \qquad
 \left(x\in \mathbb{K}\right),
\]
from which either $F(1, 1)=1$ or $F(x, 1)=0$ follows for any $x\in
\mathbb{K}$.

Using that
\[
 0= \mathcal{F}(x, x, 1, 1)=
 F(x^2 , 1)-F\left(1 , 1\right)\,F\left(x , x\right)+2F \left(x , x\right)-2F^2\left(x , 1\right)
 \qquad
 \left(x\in \mathbb{K}\right),
\]
we obtain that
\[
 \left(F(1, 1)-2\right)F\left(x , x\right)= F(x^2 , 1)-2F^2\left(x , 1\right)
 \qquad
 \left(x\in \mathbb{K}\right).
\]
Now, if $F(1, 1)=0$, then according to the above identities $ F(x,
1)=0 $ would follow for all $x\in \mathbb{K}$. Since $\mathcal{F}(x,
x, 1, 1)=0$ is also fulfilled by any $x\in \mathbb{K}$, this
immediately implies that
\[
 -2f(x)=-2F(x, x)= F(x^{2}, 1)-F(x, 1)^{2}=0
 \qquad
 \left(x\in \mathbb{K}\right),
\]
i.e., $f$ is identically zero.

In case $F(1, 1)\neq 0$, then necessarily $F(1, 1)=1$ from which
\[
 -F(x, x)= F(x^{2}, 1)-2F(x, 1)^{2}
 \qquad
 \left(x\in \mathbb{K}\right).
\]
Define the non-identically zero additive function $a\colon
\mathbb{K}\to \mathbb{C}$ by
\[
 a(x)=F(x, 1)
 \qquad
 \left(x\in \mathbb{K}\right)
\]
to get that
\[
 f(x)=F(x, x)= -F(x^{2}, 1)+2F(x, 1)^{2}= 2a(x)^{2} -a(x^{2})
 \qquad
 \left(x\in \mathbb{K}\right).
\]
Since $\mathcal{F}(x, x, x, x)=0$ has to hold, the additive function
$a\colon \mathbb{K}\to \mathbb{C}$ has to fulfill identity
\[
 -a(x^4)+a^2(x^2)+4a^2(x)a(x^2)-4a^4(x)=0
 \qquad
 \left(x\in \mathbb{K}\right)
\]
too.

In what follows, we will show that the additive function $a$ is of a
rather special form.

Indeed,
\[
 0= \mathcal{F}(x, y, z, 1)
 \qquad
 \left(x, y, z\in \mathbb{K}\right)
\]
means that $a$ has to fulfill equation
\[
 a(x)a(yz)+a(y)a(xz)+a(z)a(xy)= 2a(x)a(y)a(z)+a(xyz)
 \qquad
 \left(x, y, z\in \mathbb{K}\right)
\]
Let now $z^{\ast}\in \mathbb{K}$ be arbitrarily fixed to have
\[
 a(x)a(yz^{\ast})+a(y)a(xz^{\ast})+a(z^{\ast})a(xy)= 2a(x)a(y)a(z^{\ast})+a(xyz^{\ast})
 \qquad
 \left(x, y, z\in \mathbb{K}\right).
\]
Define the additive function $A\colon \mathbb{K}\to \mathbb{C}$ by
\[
 A(x)=a(xz^{\ast})-a(z^{\ast})a(x)
 \qquad
 \left(x\in \mathbb{K}\right)
\]
to receive that
\[
 A(xy)=a(x)A(y)+a(y)A(x)
 \qquad
 \left(x, y\in \mathbb{K}\right),
 \]
 which is a special convolution type functional equation.
Due to Theorem 12.2 of \cite{Sze91}, we get that
\begin{enumerate}[(a)]
 \item the function $A$ is identically zero, implying that $a$ has to be multiplicative.
 Note that $a$ is additive, too. Thus, for the quadratic mapping $f\colon \mathbb{K}\to \mathbb{C}$ there exists a
 homomorphism $\varphi\colon \mathbb{K}\to \mathbb{C}$ such that
 \[
  f(x)=\varphi(x)^{2}
  \qquad
  \left(x\in \mathbb{K}\right).
 \]

\item or there exists multiplicative functions $m_{1}, m_{2}\colon \mathbb{K}\to \mathbb{C}$
and a complex constant $\alpha$ such that
\[
 a(x)=\frac{m_{1}(x)+m_{2}(x)}{2}
 \qquad
 \left(x\in \mathbb{K}\right)
\]
and
\[
 A(x)=\alpha \left(m_{1}(x)-m_{2}(x)\right)
  \qquad
 \left(x\in \mathbb{K}\right).
\]
Due to the additivity of $a$, in view of the definition of the
mapping $A$, we get that $A$ is additive, too.

This however means that both the maps $m_{1}+m_{2}$ and
$m_{1}-m_{2}$ are additive, from which the additivity of $m_{1}$ and
$m_{2}$ follows, yielding that they are in fact homomorphisms.

Since
\[
 F(x, x)= f(x)= 2a(x)^{2}-a(x^{2})
 \qquad
 \left(x\in \mathbb{K}\right),
\]
we obtain for the quadratic function $f\colon \mathbb{K}\to
\mathbb{C}$ that there exist homomorphisms $\varphi_{1},
\varphi_{2}\colon \mathbb{K}\to \mathbb{C}$ such that
\[
 f(x)=\varphi_{1}(x)\varphi_{2}(x)
 \qquad
 \left(x\in \mathbb{K}\right).
\]
\end{enumerate}

Summing up, we received the following: identity
\[
 f(x^{2})= f(x)^{2}
 \qquad
 \left(x\in \mathbb{K}\right)
\]
holds for the quadratic function $f\colon \mathbb{K}\to \mathbb{C}$
if and only if there exists homomorphisms $\varphi_{1},
\varphi_{2}\colon \mathbb{K}\to \mathbb{C}$ such that
\[
 f(x)=\varphi_{1}(x)\varphi_{2}(x)
 \qquad
 \left(x\in \mathbb{K}\right).
\]
\end{opp}

\begin{opp}[Not necessarily additive solutions]
 Motivated by the above open problem as well as Remark \ref{rem_noadd}, we can also pose the question below.

Let $n\in \mathbb{N}$ be arbitrary, $\mathbb{K}$ a field, $f_{1},
\ldots, f_{n}\colon \mathbb{K}\to \mathbb{C}$ be generalized or
exponential polynomials. Suppose further that we are given natural
numbers $p_{1}, \ldots, p_{n}, q_{1}, \ldots, q_{n}$ so that they
fulfill condition $(\mathscr{C})$. Suppose also that equation
\begin{equation}\label{eqfin}
 \sum_{i=1}^{n}f^{q_{i}}_{i}\left(x^{p_{i}}\right)=0
\end{equation}
is satisfied.

In Example \ref{ex3} we gave a solution of consisting non-additive
solutions of \eqref{eqreex1}. Namely,
\begin{align*}
    f(x)&=-(20+4a(x)+a^2(x))x,\\
    g(x)&=2(1+a(x))x,\\
    h(x)&=2x
\end{align*}
The functions $f,g$ and $h$ are exponential polynomial solutions of
\eqref{eqreex1}. Thus it is clear that Theorem \ref{thmph} do not
hold without additivity of $f,g$ and $h$.

Again, the question is, how can we characterize (exponential)
polynomial solutions of \eqref{eqfin}?
\end{opp}



%

\begin{opp}[Solutions on rings and on fields of finite characteristic]
 As it already appears in the definition of homomorphisms, the natural domain and also the natural range of the functions in \eqref{eqreex1}
 are rings.

On the other hand, it is easy to see that there is no nontrivial field
homomorphism from $\mathbb{K}$ to $\mathbb{C}$, if the
characteristic of $\mathbb{K}$ is finite. The careful reader can
also deduce using our methods, that already \eqref{eqreex1} has
no solution in this case. At the same time, it can be easily seen
that the equation has solutions if the functions $f_i\colon \mathbb{K}\to \mathbb{L}$
are constant multiple of a field homomorphism where
$\mathbb{K}$ and $\mathbb{L}$ has the same characteristic.

According to this, the general question arises how the solutions of
equation \eqref{Eq1.3} look like in case when the functions $f_{1},
\ldots, f_{n}$ are defined between (not necessarily commutative)
rings?
\end{opp}

\begin{opp}[Regular solutions in case $\mathbb{K}=\mathbb{C}$]

To pose our last open problem, here we recall the following.
Concerning homomorphisms, instead of $\mathbb{R}$, in $\mathbb{C}$
the situation is completely different, see Kestelman \cite{Kes51},
since we have the following.

\begin{prop}
 The only \emph{continuous} endomorphisms $f\colon \mathbb{C}\to \mathbb{C}$ are
 $f\equiv 0$, $f\equiv \mathrm{id}$ or
 \[
  f(x)=\overline{x}
  \qquad
  \left(x\in \mathbb{C}\right).
 \]
These endomorphisms are referred to as \emph{trivial endomorphisms}.
\end{prop}

Concerning \emph{nontrivial endomorphisms} we quote here the
following.

\begin{prop}
 \begin{enumerate}[(i)]
  \item There exist nontrivial automorphisms of $\mathbb{C}$.
  \item If $f\colon \mathbb{C}\to \mathbb{C}$ is a nontrivial automorphism, then $f\vert_{\mathbb{R}}$ is discontinuous.
  \item If $f\colon \mathbb{C}\to \mathbb{C}$ is a nontrivial automorphism, then the closure of the set $f(\mathbb{R})$ is the whole complex plane.
  \item If $f\colon \mathbb{C}\to \mathbb{C}$ is a nontrivial automorphism, then $f(\mathbb{R})$ is a proper subfield of $\mathbb{C}$,
  $\mathrm{card}(f(\mathbb{R}))=\mathfrak{c}$ and either the planar (Lebesgue) measure of $f(\mathbb{R})$ is zero or $f(\mathbb{R})\subsetneq \mathbb{C}$ is a
  saturated non-measurable set.
 \end{enumerate}

\end{prop}

 As we saw above the continuous endomorphisms of $\mathbb{C}$ are of really pleasant form.
 This immediately implies that the \emph{continuous} solutions of equation \eqref{Eq1.3}
 in case $\mathbb{K}=\mathbb{C}$ also have the same beautiful structure.

 Obviously, the continuity assumption can be weakened to guarantee the same result.
 At the same time, our question is whether instead of a regularity assumption,
 an additional \emph{algebraic} supposition for the unknown functions would imply the same?
 \end{opp}

\noindent \textbf{Acknowledgement.} The authors are grateful to
\emph{Professors Mikl\'os Laczkovich} and \emph{L\'aszl\'o
Sz\'e\-kely\-hi\-di} for their valuable remarks and also for their
helpful discussions.

The research of the first author has been supported by the Hungarian
Scientific Research Fund (OTKA) Grant K 111651. The publication is
also supported by the EFOP-3.6.1-16-2016-00022 project. The project
is co-financed by the European Union and the European Social Fund.
The second author was supported by the internal research project
R-AGR-0500 of the University of Luxembourg and by the Hungarian
Scientific Research Fund (OTKA) K104178. The third author was
supported by EFOP 3.6.2-16-2017-00015.


\begin{thebibliography}{10}

\bibitem{Dar77}
Jean~Gaston Darboux.
\newblock Sur la composition des forces en statique.
\newblock {\em Bull. Acad. Polon. Sci. Sér. Sci. Math. Astronom. Phys.},
  9:281--288, 1877.

\bibitem{Hal00}
Franz Halter-Koch.
\newblock Characterization of field homomorphisms and derivations by functional
  equations.
\newblock {\em Aequationes Math.}, 59(3):298--305, 2000.

\bibitem{HalRei00}
Franz Halter-Koch and Ludwig Reich.
\newblock Characterization of field homomorphisms by functional equations.
\newblock {\em Publ. Math. Debrecen}, 56(1-2):179--183, 2000.

\bibitem{HalRei01}
Franz Halter-Koch and Ludwig Reich.
\newblock Characterization of field homomorphisms by functional equations.
  {II}.
\newblock {\em Aequationes Math.}, 62(1-2):184--191, 2001.

\bibitem{Kes51}
Hyman Kestelman.
\newblock Automorphisms of the field of complex numbers.
\newblock {\em Proc. London Math. Soc. (2)}, 53:1--12, 1951.

\bibitem{KL1}
Gergely Kiss and Mikl\'os Laczkovich.
\newblock Linear functional equations, differential operators and spectral
  synthesis.
\newblock {\em Aequationes Math.}, 89(2):301--328, 2015.

\bibitem{KL2}
Gergely Kiss and Mikl\'os Laczkovich.
\newblock Derivations and differential operators on rings and fields.
\newblock {\em Annales Univ. Sci. Budapest., Sect. Comp.}, 48(31--43), 2018.

\bibitem{KisVar13}
Gergely Kiss and Adrienn Varga.
\newblock Existence of nontrivial solutions of linear functional equations.
\newblock {\em Aequationes Math.}, 88(1-2):151--162, 2014.

\bibitem{Kuc09}
Marek Kuczma.
\newblock {\em An introduction to the theory of functional equations and
  inequalities}.
\newblock Birkh\"{a}user Verlag, Basel, second edition, 2009.
\newblock Cauchy's equation and Jensen's inequality, Edited and with a preface
  by Attila Gil\'{a}nyi.

\bibitem{L14}
Mikl\'os Laczkovich.
\newblock Local spectral synthesis on abelian groups, 142 (2014),.
\newblock {\em Acta Math. Hungar.}, 142(2):313--329, 2014.

\bibitem{SchRei84}
Ludwig Reich and Jens Schwaiger.
\newblock On polynomials of additive and multiplicative functions.
\newblock {\em Functional Equations: History, Application and Theory}, pages
  127--160, 1984.

\bibitem{Shu10}
Ekaterina Shulman.
\newblock Decomposable functions and representations of topological semigroups.
\newblock {\em Aequationes Math.}, 79(1-2):13--21, 2010.

\bibitem{Sze91}
L\'aszl\'o Sz\'ekelyhidi.
\newblock {\em Convolution type functional equations on topological abelian
  groups}.
\newblock World Scientific Publishing Co., Inc., Teaneck, NJ, 1991.

\bibitem{Sze06}
László Székelyhidi.
\newblock {\em Discrete spectral synthesis and its applications}.
\newblock Springer Monographs in Mathematics. Springer, Dordrecht, 2006.

\end{thebibliography}

\end{document}